\documentclass{amsart}
\usepackage{amssymb,mathrsfs}
\newtheorem{theorem}{Theorem}[section]
\newtheorem{lemma}[theorem]{Lemma}
\newtheorem{corollary}[theorem]{Corollary}
{\theoremstyle{definition}
\newtheorem{remark}[theorem]{Remark}}
\title{Sums of three squares in $\mathbf{Q}(\sqrt{3})$, and in $\mathbf{Q}(\sqrt{17})$}
\author{Shigeaki Tsuyumine}
\address{Mie University, Faculty of Education, Department of Mathematics, Tsu, 514-8507, Japan}
\email{tsuyu@edu.mie-u.ac.jp}
\allowdisplaybreaks
\begin{document}

\begin{abstract}The numbers of representations of totally positive integers as sums of three integer squares in $\mathbf{Q}(\sqrt{3})$ and in $\mathbf{Q}(\sqrt{17})$, are studied by using Shimura lifting map of Hilbert modular forms. We show the following results.  In case of $\mathbf{Q}(\sqrt{3})$, a totally positive integer $a+b\sqrt{3}$ is represented as a sum of three integer squares if and only if $b$ is even. In case of $\mathbf{Q}(\sqrt{17})$, a totally positive integer is represented as a sum of three integer squares if and only if it is not in the form $\pi_{2}^{2e}\pi_{2}'^{2e'}\mu$ with $\mu\equiv7\pmod{\pi_{2}^{3}}$ or $\mu\equiv7\pmod{\pi_{2}'^{3}}$ where $\pi_{2},\pi_{2}'$ are prime elements with $2=\pi_{2}\pi_{2}'$. A similar result as Gauss's three squares theorem in both cases of $\mathbf{Q}(\sqrt{3})$ and $\mathbf{Q}(\sqrt{17})$, and as its application, tables of class numbers of their totally imaginary extensions are given.
\end{abstract}

\maketitle
\section{Introduction}
H.~Maass \cite{Maass} showed that all totally positive integers in $\mathbf{Q}(\sqrt{5})$ are expressed as a sum of three integer squares in $\mathbf{Q}(\sqrt{5})$ by giving the explicit formula for the numbers $r_{3,\mathbf{Q}(\sqrt{3})}(\alpha)$ of representations for totally positive integers $\alpha$. The formula involves class numbers of totally imaginary quadratic extensions $\mathbf{Q}(\sqrt{5},\sqrt{-\alpha})$ of $\mathbf{Q}(\sqrt{5})$. Since the numbers of representations are calculable, by using Maass's formula H.~Cohn \cite{Cohn} gave a table of class numbers of totally imaginary quadratic extensions of $\mathbf{Q}(\sqrt{5})$. In our previous paper \cite{Tsuyumine4} we showed, by using Shimura lifting map, that any totally positive integer $a+b\sqrt{2}\ (a,b\in\mathbf{Z})$ in $\mathbf{Q}(\sqrt{2})$ is expressed as a sum of three integer squares if and only if $b$ is even. Further we gave the formula for the number $r_{3,\mathbf{Q}(\sqrt{2})}(a+b\sqrt{2})$ representing  $a+b\sqrt{2}$ as sums of three integer squares which involves the class number of  a totally imaginary quadratic extension $\mathbf{Q}(\sqrt{2},\sqrt{-a-b\sqrt{2}})$ of $\mathbf{Q}(\sqrt{2})$.

In the present paper we are concerned with the fields $\mathbf{Q}(\sqrt{3})$ and $\mathbf{Q}(\sqrt{17})$. By using Shimura lifting map of Hilbert modular forms again, we give the following results. A totally positive integer $a+b\sqrt{3}\in \mathbf{Q}(\sqrt{3})$ is expressed as a sum of three integer squares if and only if $b$ is even. The formula for $r_{3,\mathbf{Q}(\sqrt{3})}(\alpha\nu^{2})$ with totally positive square-free $\alpha$ and with products $\nu$ of totally positive prime elements is obtained, which gives as its application, a table of class numbers of totally imaginary quadratic extensions of $\mathbf{Q}(\sqrt{3})$. As for $\mathbf{Q}(\sqrt{17})$, let $\pi_{2}:=\tfrac{5+\sqrt{17}}{2},\pi_{2}':=\tfrac{5-\sqrt{17}}{2}$. A totally positive integer is expressed as a sum of three integer squares if and only if it is not in the form $\pi_{2}^{2e}\pi_{2}'^{2e'}\mu$ with nonnegative integers $e,e'$ and with $\mu\equiv7\pmod{\pi_{2}^3}$ or $\mu\equiv7\pmod{\pi_{2}'^3}$. The formula for $r_{3,\mathbf{Q}(\sqrt{17})}$ is obtained as well as a table of class numbers of totally imaginary quadratic extensions of $\mathbf{Q}(\sqrt{17})$. 

To see how the Shimura lifting map works, we illustrate the case of rational integers (see \cite{Tsuyumine1,Tsuyumine2,Tsuyumine3}). Let $\mathfrak{H}$ be the upper half plane $\{z\in\mathbf{C}\mid \Im z>0\}$, and let $\mathbf{e}(z):=\exp(2\pi\sqrt{-1}z)$. For a discriminant $D$ of a quadratic field, $\chi_{D}$ denotes the Kronecker-Jacobi-Legendre symbol. If $a$ is a square-free natural number, then $a^{\ast}$ denotes $a$ or $4a$ according as $a\equiv1\pmod{4}$ or not. For $N\in\mathbf{N}$, $(\mathbf{Z}/N)^{\ast}$ denotes the group of Dirichlet characters modulo $N$ with $\mathbf{1}_{N}$ as the identity element. Let $\mathbf{M}_{k+1/2}(N,\chi)$ be the space of holomorphic modular forms for $\Gamma_{0}(N)$ of weight $k+1/2$ with character $\chi\in(\mathbf{Z}/N)^{\ast}\ (4|N)$. Then there is a Shimura lifting map $\mathscr{S}_{a^{\ast},\chi}$ of $\mathbf{M}_{k+1/2}(N,\chi)$ to $\mathbf{M}_{2k}(N/2,\chi^{2})$ for each square-free $a\ge1$ and for $k\ge1$, such as 
\begin{align}
\mathscr{S}_{a^{\ast},\chi}(f)=C_{f,a}+\sum_{1\le n}\sum_{0<d|n}(\chi_{a^{\ast}}\chi)(d)d^{k-1}c_{an^{2}/d^{2}}\mathbf{e}(nz)  \label{eqn:s-l-e}
\end{align}
with $f(z)=\sum_{0\le n}c_{n}\mathbf{e}(nz)$, where $C_{f,a}$ is a constant for which the left hand side of (\ref{eqn:s-l-e}) is a modular form. Let $\theta(z)$ is a theta series, namely $\theta(z)=\sum_{n=-\infty}^{\infty}\mathbf{e}(n^{2}z)$. If $f$ is a product of $\theta(z)$ and an Eisenstein series $G(z)$ of weight $k$, then $C_{f,a}$ is obtained as the constant term of a Hilbert-Eisenstein series over a field $\mathbf{Q}(\sqrt{a})$ where the Hilbert-Eisenstein series is determined only by $G(z)$.

Let $r_{3,\mathbf{Q}}(n)$ denote the numbers of representations of $n$ as sums of three rational integral squares. The generating function $\sum_{0\le n}r_{3,\mathbf{Q}}(n)\mathbf{e}(nz)$ of $r_{3,\mathbf{Q}}(n)$ is $\theta(z)^{3}\in\mathbf{M}_{3/2}(4,\chi_{-4})$. Let $n=am^{2}$ with square-free $a>0$. Since $\theta(z)^{2}$ is expressed as an Eisenstein series, $\theta(z)^{3}$ is a product of $\theta(z)$ an the Eisenstein series. The constant term of $\mathscr{S}_{a^{\ast},\chi}(\theta(z)^{3})$ is obtained from the constant term of an associated Hilbert-Eisenstein series which involves the class number of $\mathbf{Q}(\sqrt{-a})$. Since $\mathbf{M}_{2}(2,\mathbf{1}_{2})$ is one dimensional, it is spanned by an Eisenstein series of weight $2$, and by comparing the Fourier coefficients of $\mathscr{S}_{a^{\ast},\chi}(\theta(z)^{3})$ and those of the Eisenstein series, we obtain the formula for $r_{3,\mathbf{Q}}(n)$ (\cite{Tsuyumine1}). In particular Gauss's three squares theorem is derived from the formula, that is $r_{3,\mathbf{Q}}(a)=2^{2}3h_{\mathbf{Q}(\sqrt{-a})}$ for square-free $a\equiv1,2\pmod{4}$, and $r_{3,\mathbf{Q}}(a)=2^{3}3h_{\mathbf{Q}(\sqrt{-a})}$ for square-free $a\equiv3\pmod{8}$, $h_{\mathbf{Q}(\sqrt{-a})}$ being the class number of $\mathbf{Q}(\sqrt{-a})$. The one of purposes of the present paper is to obtain such a kind of  formulas on $\mathbf{Q}(\sqrt{3})$ or on $\mathbf{Q}(\sqrt{17})$ instead of $\mathbf{Q}$.

Our paper \cite{Tsuyumine4} gives only a partial result on Shimura lifts of noncuspidal Hilbert modular forms over a totally real field $K$, however it is shown that products of theta series and Hilbert-Eisenstein series over $K$ have Shimura lifts whose constant terms are obtained from Hilbert-Eisenstein series over totally real quadratic extensions of $K$. So there is a way that a similar argument as in the elliptic modular forms  can be made for Hilbert modular forms. In the present paper we carry out this for $\mathbf{Q}(\sqrt{3})$ and $\mathbf{Q}(\sqrt{17})$. However while the space $\mathbf{M}_{2}(2,\mathbf{1}_{2})$ contains no nontrivial cusp forms in the elliptic modular case as well as the cases of $\mathbf{Q}(\sqrt{5})$ and $\mathbf{Q}(\sqrt{2})$, there is a non-trivial cusp form of weight $2$ in each case of $\mathbf{Q}(\sqrt{3})$ or $\mathbf{Q}(\sqrt{17})$, and a closer study of Fourier coefficients of Hilbert modular forms is necessary.

\section{Notations and some preceding results}\label{sect:n-r}
Let $K$ be a totally real algebraic number field of degree $g$ over $\mathbf{Q}$. We denote by $\mathcal{O}_{K}$, $\mathfrak{d}_{K}$ and $d_{K}$, ring of algebraic integers, the different of $K$ and the discriminant respectively.  For $\alpha\in K$, $\alpha^{(1)},\cdots,\alpha^{(g)}$ denotes the conjugates of $\alpha$ in a fixed order. If $\alpha^{(i)}$ is positive for every $i$, then we call $\alpha$ {\itshape totally positive}, and denote it by $\alpha\succ 0$. For an integral ideal $\mathfrak{N}$, let $\mathcal{E}_{\mathfrak{N}}$ denotes the group of totally positive units congruent to $1$ modulo $\mathfrak{N}$, and so $\mathcal{E}_{\mathcal{O}_{K}}$ denotes the group of all totally positive units, while $\mathcal{O}_{K}^{\times}$ denotes the group of all units. We denote by $\mathrm{N}$ and $\mathrm{tr}$, the norm map and the trace map of $K$ over $\mathbf{Q}$ respectively, namely $\mathrm{N}(\alpha)=\prod_{i=1}^{g}\alpha^{(i)}$ and $\mathrm{tr}(\alpha)=\sum_{i=1}^{g}\alpha^{(i)}$. An integral ideal is called {\itshape odd} or {\itshape even} according as it is coprime to the ideal $(2)$ or not. An integer $\alpha$ in $\mathcal{O}_{K}$ is called {\itshape odd} or {\itshape even} according as the ideal $(\alpha)$ is odd or even. If $\mathfrak{P}$ is a prime ideal, then $v_{\mathfrak{P}}$ denotes the $\mathfrak{P}$-adic valuation. Let $\mu_{K}$ denote the M\"obius function on $K$, and let $\varphi_{K}$ denote the Euler function on $K$, that is, $\varphi_{K}(\mathfrak{N})=\mathrm{N}(\mathfrak{N})\prod_{\mathfrak{P}|\mathfrak{N}}(1-\mathrm{N}(\mathfrak{P})^{-1})$ for an integral ideal $\mathfrak{N}$. We denote by $C_{\mathfrak{N}}$, the fractional ideal class group modulo $\mathfrak{N}$ in the narrow sense where ideals have denominators coprime to $\mathfrak{N}$, and denote by $C_{\mathfrak{N}}^{\ast}$, the group of characters of $C_{\mathfrak{N}}$. An element of $C_{\mathfrak{N}}^{\ast}$ is called a (classical) Hecke character. The identity element of $C_{\mathfrak{N}}^{\ast}$ is denoted by $\mathbf{1}_{\mathfrak{N}}$, for which $\mathbf{1}_{\mathfrak{N}}(\mathfrak{A})$ is $1$ or $0$ according as a numerator of an ideal $\mathfrak{A}$ is coprime to $\mathfrak{N}$ or not. If $\mathfrak{N}$ is principal with a generator $\nu$, then we denote $\mathbf{1}_{\mathfrak{N}}$ also by $\mathbf{1}_{\nu}$. For $\mathbf{a}=(a_{1},\cdots,a_{g})\in\{0,1\}^{g}$, we define $\mathrm{sgn}^{\mathbf{a}}(\alpha)$ by setting $\mathrm{sgn}^{\mathbf{a}}(\alpha):=\mathrm{sgn}(\alpha^{(1)})^{a_{1}}\cdots\,\mathrm{sgn}(\alpha^{(g)})^{a_{g}}$ for $\alpha\in K,\ne0$ where $\mathrm{sgn}^{\mathbf{a}}(0):=1$. For $\psi\in C_{\mathfrak{N}}^{\ast}$, let $\mathbf{e}_{\psi}=(e_{1},\cdots,e_{n})\in\{0,1\}^{g}$ be so that $\psi(\mu)=\mathrm{sgn}^{\mathbf{e}_{\psi}}(\mu)$ for $\mu\equiv1\bmod{\mathfrak{N}},\mu\ne 0$. The character $\psi$ is called {\itshape even} if $\mathbf{e}_{\psi}=(0,\cdots,0)$, and it is called {\itshape odd} if $\mathbf{e}_{\psi}=(1,\cdots,1)$. 

For a Hecke character $\psi$, we  denote by $\mathfrak{f}_{\psi}$, the conductor of $\psi$, and put $\mathfrak{e}_{\psi}:=\mathfrak{f}_{\psi}\prod_{\psi(\mathfrak{P})=0,\mathfrak{P}\nmid\mathfrak{f}_{\psi}}\mathfrak{P}$. The primitive character associated with $\psi$ is denoted by $\widetilde{\psi}$. For an integral ideal $\mathfrak{N}$, we define $\psi_{\mathfrak{N}}$ to be $\widetilde{\psi}\mathbf{1}_{\mathfrak{N}}$. For an integer $\alpha\in K$ not square, we denote by $\mathfrak{f}_{\alpha}$ or $d_{K(\sqrt{\alpha})/K}$, the relative discriminant of $K(\sqrt{\alpha})$ over $K$, and denote by $\psi_{\alpha}\in C_{\mathfrak{f}_{\alpha}}^{\ast}$, the character associated with the extension, where $\psi(\mathfrak{P})$ is $1,-1$ or $0$ according as the prime $\mathfrak{P}$ of $K$ is split at $K(\sqrt{\alpha})$, inert or ramified.

We need some more general character than a Hecke character. For the purpose it is convenient to use the idelic language. For a prime $\mathfrak{P}$ of $K$, let $K_{\mathfrak{P}},\mathcal{O}_{\mathfrak{P}},\mathcal{O}_{\mathfrak{P}}^{\times}$ be the $\mathfrak{P}$-adic completion of $K$, the maximal local ring and the group of its units. For an integral ideal  $\mathfrak{N}$ of $K$, let $J(\mathfrak{N})$ denote the group consisting of ideles whose $\mathfrak{P}$-th components are in  $\mathcal{O}_{\mathfrak{P}}$ for $\mathfrak{P}|\mathfrak{N}$. Put $U_{K}:=\prod_{\mathfrak{P}}\mathcal{O}_{\mathfrak{P}}^{\times}\times(\mathbf{R}^{+})^{g}$ with $\mathbf{R}^{+}=\{ x\in\mathbf{R}\mid x>0\}$. Let  $K^{\times}(\mathfrak{N}):=K\cap J(\mathfrak{N})$, namely $K^{\times}(\mathfrak{N})$ is the group of elements in $K^{\times}$ whose denominators and numerators are both coprime to $\mathfrak{N}$.  Let $K_{\mathfrak{N}}^{\times}$ denote the subgroup of $K^{\times}(\mathfrak{N})$ consisting of totally positive elements multiplicatively congruent to $1$ modulo $\mathfrak{N}$. A homomorphism of the finite idele to $C_{\mathfrak{N}}$ by sending $j=(j_{\mathfrak{P}})$ to an ideal  class containing fractional ideal $\prod_{\mathfrak{P}}\mathfrak{P}^{v_{\mathfrak{P}}(j_{p})}$ gives the natural isomorphism between $J(\mathfrak{N})/(K_{\mathfrak{N}}^{\times}U_{K})$ and $C_{\mathfrak{N}}$, and we identify them. Let  $U_{\mathfrak{N}}$ be the subgroup of $U_{K}$ consisting of ideles whose $\mathfrak{P}$-th components are in $1+\mathfrak{P}^{v_{\mathfrak{P}}(\mathfrak{N})}\mathcal{O}_{\mathfrak{P}}$. The factor group $U_{K}/U_{\mathfrak{N}}$ is isomorphic to $(\mathcal{O}_{K}/\mathfrak{N})^{\times}$. We fix local parameters $\varpi_{\mathfrak{P}}$ of $\mathcal{O}_{\mathfrak{P}}$ for $\mathfrak{P}$, which give the isomorphism of $U_{K}/U_{\mathfrak{N}}$ onto $(\mathcal{O}_{K}/\mathfrak{N})^{\times}$. Let $(O_{K}/\mathfrak{N})^{\ast}$ denote the group of characters of $(O_{K}/\mathfrak{N})^{\times}$. For $\phi\in C_{\mathfrak{N}}^{\ast}$ and for $\omega\in(O_{K}/\mathfrak{N})^{\ast}$ we define
\begin{align*}
\psi(\xi\cdot\mathfrak{A})=\omega(\xi\prod_{\mathfrak{P}|\mathfrak{N}}\varpi_{\mathfrak{P}}^{v_{\mathfrak{P}}(\mathfrak{A})})\phi((\xi)\mathfrak{A})
\end{align*}
for $\xi\in K^{\times}$ and for an ideal $\mathfrak{A}$ where a denominator of $\xi\mathfrak{A}$ is coprime to $\mathfrak{N}$. We define $\psi(\xi\cdot\mathfrak{A})=0$ if a numerator of $\xi\mathfrak{A}$ is not coprime to $\mathfrak{N}$. Obviously there holds an equality $\psi(\xi\cdot\mathfrak{A})\psi(\xi'\cdot\mathfrak{A}')=\psi(\xi\xi'\cdot\mathfrak{AA}')$. This $\psi$ is a character of  $J(\mathfrak{N})/(K_{\mathfrak{N}}^{\times}U_{\mathfrak{N}})$. Since  $J(\mathfrak{N})/(K_{\mathfrak{N}}^{\times}U_{\mathfrak{N}})\simeq(U_{K}/U_{\mathfrak{N}})^{\times}\times J(\mathfrak{N})/(K_{\mathfrak{N}}^{\times}U_{K})\simeq(O_{K}/U_{\mathfrak{N}})^{\times}\times C_{\mathfrak{N}}$ as groups, the group $(J(\mathfrak{N})/(K_{\mathfrak{N}}^{\times}U_{\mathfrak{N}}))^{\ast}$ of characters is isomorphic to $(O_{K}/U_{\mathfrak{N}})^{\ast}\times C_{\mathfrak{N}}^{\ast}$. We put $\psi(\xi):=\psi(\xi\cdot\mathcal{O}_{K})=\omega(\xi)\phi((\xi))$ for $\xi\in K^{\times}(\mathfrak{N})$, and $\psi(\mathfrak{A}):=\psi(1\cdot\mathfrak{A})$ for an ideal $\mathfrak{A}$. Also for $\psi\in(O_{K}/U_{\mathfrak{N}})^{\ast}\times C_{\mathfrak{N}}^{\ast}$, $\mathbf{e}_{\psi}$ is defined as in the case of Hecke characters, however the equality $\mathbf{e}_{\psi}=\mathbf{e}_{\phi}$ holds since $\omega(\xi)=1$ for $\xi$ with $\xi\equiv1\pmod{\mathfrak{N}}$. We call $\psi$ {\itshape primitive} if there is no integral ideal $\mathfrak{M}$ so that $\psi=\psi'\mathbf{1}_{\mathfrak{N}}$ with a character $\psi'$ of $J(\mathfrak{M})/(K_{\mathfrak{M}}^{\times}U_{\mathfrak{M}})$. In such a case we denote $\mathfrak{N}$ by $\mathfrak{f}_{\psi}$, and call it the {\itshape conductor} of $\psi$, and put $\mathfrak{e}_{\psi}:=\mathfrak{f}_{\psi}\prod_{\psi(\mathfrak{P})=0,\mathfrak{P}\nmid\mathfrak{f}_{\psi}}\mathfrak{P}$. For $\psi$ primitive, we define the Gauss sum by
\begin{align}
\tau_{K}(\psi):=\psi(\mu\cdot\mathfrak{f}_{\psi}\mathfrak{d}_{K})\sum_{\xi\in\mathcal{O}_{K}(\mathrm{mod}\,\mathfrak{f}_{\psi})\atop \xi\succ0}\psi(\xi)\mathbf{e}(\mathrm{tr}(\mu\xi))\label{eqn:GS}
\end{align}
with $\mu\in K,\ \succ0,\ (\mu\mathfrak{f}_{\psi}\mathfrak{d}_{K},\mathfrak{f}_{\psi})=\mathcal{O}_{K}$, where $\tau_{K}(\psi)$ is determined up the choices of $\mu$. The standard argument shows that $|\tau_{K}(\psi)|=\mathrm{N}(\mathfrak{f}_{\psi})^{1/2}$, and that $\tau_{K}(\psi)\tau_{K}(\overline{\psi})=\mathrm{sgn}^{\mathbf{e}_{\psi}}(-1)\psi(-1)\mathrm{N}(\mathfrak{f}_{\psi})$. When $\omega$ is trivial, $\tau_{K}(\psi)$ of (\ref{eqn:GS}) coincides with the Gauss sum of a Hecke character.

Let $\mathfrak{N},\mathfrak{N}'$ be integral ideals. Let $\psi=\omega\phi\in(\mathcal{O}_{K}/\mathfrak{N})^{\ast}\times C_{\mathfrak{N}}^{\ast}, \psi'=\omega'\phi'\in(\mathcal{O}_{K}/\mathfrak{N}')^{\ast}\times C_{\mathfrak{N}'}^{\ast}$ so that $\psi\overline{\psi}'$ is an even or odd Hecke character in $C_{\mathfrak{NN}'}^{\ast}$, namely
\begin{align}
\mathbf{e}_{\psi\overline{\psi}'}=(0,\cdots,0)\mbox{ or }(1,\cdots,1)\mbox{, and }\omega(\xi)=\omega'(\xi)\mbox{ for }\xi\in K^{\times}((\mathfrak{N},\mathfrak{N}')).\label{cond:psi}
\end{align}
For a fractional ideal $\mathfrak{M}$ and for a totally positive $\nu \in K$, we define
\begin{align}
\sigma_{k-1,\psi}^{\psi'}(\nu;\mathfrak{M}):=\sum_{\nu\mathfrak{M}\subset\mathfrak{A}\subset\mathcal{O}_{K}}\psi(\mathfrak{A})\psi'(\nu\cdot\mathfrak{M}\mathfrak{A}^{-1})\mathrm{N}(\mathfrak{A})^{k-1},\label{def:sigma}
\end{align}
where it is $0$ if $\nu\mathfrak{M}$ is not integral. We also define $\sigma_{k-1,\psi}^{\psi'}(\nu;\mathfrak{M})$ to be $0$ if $\nu$ is not totally positive.  By a condition (\ref{cond:psi}), the summation (\ref{def:sigma}) is equal to $\omega'(\nu\prod_{\mathfrak{P}|\mathfrak{N}'}\varpi_{\mathfrak{P}}^{-v_{\mathfrak{P}}(\nu)})$ $\sum_{\nu\mathfrak{M}\subset\mathfrak{A}\subset\mathcal{O}_{K}}\phi(\mathfrak{A})\phi'(\nu\mathfrak{M}\mathfrak{A}^{-1})\mathrm{N}(\mathfrak{A})^{k-1}$. If $\mathfrak{M}=\mathcal{O}_{K}$, then we denote $\sigma_{k-1,\psi}^{\psi'}(\nu;\mathfrak{M})$ simply by  $\sigma_{k-1,\psi}^{\psi'}(\nu)$, and further if $K$ is of class number $1$ and if both of $\psi$ and $\psi'$ are Hecke characters, namely, $\omega$ and  $\omega'$ are both trivial, then $\sigma_{k-1,\psi}^{\psi'}(\nu)$ is expressed as
\begin{align}
\sigma_{k-1,\psi}^{\psi'}(\nu)=\sum_{\delta|\nu,\delta\in\mathcal{O}_{K}/\mathcal{O}_{K}^{\times}}\psi(\delta)\psi'(\nu/\delta)\mathrm{N}((\delta))^{k-1}\hspace{1.5em}(\nu\in\mathcal{O}_{K},\succ0). \label{def:sigma1}
\end{align}
We omit $\psi$ from the notation $\sigma_{k-1,\psi}^{\psi'}$ if $\mathfrak{N}=\mathcal{O}_{K}$ and $\psi=\mathbf{1}$, and do similar as for $\psi'$.

Let $\mathfrak{H}^{g}$ be the product of $g$ copies of the upper half plane $\mathfrak{H}$. For $\gamma,\delta$ in $K$ and for $\mathfrak{z}=(z_{1},\cdots,z_{g})\in\mathfrak{H}^{g}$, $\mathrm{N}(\gamma\mathfrak{z}+\delta)$ and $\mathrm{tr}(\gamma\mathfrak{z})$ stand for $\prod_{i=1}^{g}(\gamma^{(i)}z_{i}+\delta^{(i)})$ and $\sum_{i=1}^{g}\gamma^{(i)}z_{i}$ respectively. For a matrix $A=\left(\begin{smallmatrix}\alpha&\beta\\\gamma&\delta\end{smallmatrix}\right)\in\mathrm{SL}_{2}(K)$, we put $A\mathfrak{z}=\left(\frac{\alpha^{(1)}z_{1}+\beta^{(1)}}{\gamma^{(1)}z_{1}+\delta^{(1)}},\cdots,\frac{\alpha^{(g)}z_{g}+\beta^{(g)}}{\gamma^{(g)}z_{g}+\delta^{(g)}}\right)$. We define
\begin{align*}
\Gamma_{0}(\mathfrak{D}^{-1},\mathfrak{ND}):=\{\mbox{\small$\left({\alpha\ \ \beta\atop\gamma\ \ \delta}\right)$}\in\mathrm{SL}_{2}(K)\mid\alpha,\delta\in\mathcal{O}_{K},\beta\in\mathfrak{D}^{-1},\gamma\in\mathfrak{ND}\}
\end{align*}
for a fractional ideal $\mathfrak{D}$ and for an integral ideal $\mathfrak{N}$.

Let $\psi\in(\mathcal{O}_{K}/\mathfrak{N})^{\ast}\times C_{\mathfrak{N}}^{\ast}, \psi'\in(\mathcal{O}_{K}/\mathfrak{N}')^{\ast}\times C_{\mathfrak{N}'}^{\ast}$ be as in (\ref{cond:psi}), and let $k$ be a natural number with the same parity as $\mathbf{e}_{\psi\overline{\psi}'}$. We assume that $\psi$ or $\psi'$ is nontrivial when $g=1$ and $k=2$. Then we define an Eisenstein series by
\begin{align}
G_{k,\psi}^{\psi'}(\mathfrak{z};\mathfrak{Ne}_{\psi}^{-1}\mathfrak{N}'\mathfrak{e}_{\psi'}^{-1},\mathfrak{D}):=C+2^{g}\sum_{0\prec\nu\in\mathfrak{D}}\sigma_{k-1,\psi}^{\psi'}(\nu;\mathfrak{N}^{-1}\mathfrak{e}_{\psi}\mathfrak{N}'^{-1}\mathfrak{e}_{\psi'})\mathbf{e}(\mathrm{tr}(\nu\mathfrak{z}))\label{eqn:e-s}
\end{align}
where $C$ is a constant term given by (i) $\overline{\psi}'(\mathfrak{Ne}_{\psi}^{-1}\mathfrak{f}_{\psi}\mathfrak{D})L_{K}(1-k,\psi\overline{\psi}')$ if $k>1$ and $\mathfrak{N}'=\mathcal{O}_{K}$ or if $\mathfrak{N}\subsetneq\mathcal{O}_{K}$ and  $\mathfrak{N}'=\mathcal{O}_{K}$, (ii) $\overline{\psi}(\mathfrak{N}'\mathfrak{e}_{\psi'}^{-1}\mathfrak{f}_{\psi'}\mathfrak{D})L_{K}(0,\overline{\psi}\psi')$ if $k=1,\mathfrak{N}=\mathcal{O}_{K}$ and $\mathfrak{N}'\subsetneq\mathcal{O}_{K}$, (iii) $\psi'(\mathfrak{D})L_{K}(0,\psi\overline{\psi}')+\psi(\mathfrak{D})L_{K}(0,\overline{\psi}\psi')$ if $k=1$ and $\mathfrak{N}=\mathfrak{N}'=\mathcal{O}_{K}$, and (iv) $0$ otherwise. The Eisenstein series $G_{k,\psi}^{\psi'}(\mathfrak{z};\mathfrak{Ne}_{\psi}^{-1}\mathfrak{N}'\mathfrak{e}_{\psi'}^{-1},\mathfrak{D})$ of (\ref{eqn:e-s}) is a Hilbert modular form for $\Gamma_{0}(\mathfrak{D}^{-1}\mathfrak{d}_{K}^{-1},\mathfrak{NN}'\mathfrak{D}\mathfrak{d}_{K})$ of weight $k$ with character $\psi\psi'$, namely $G_{k,\psi}^{\psi'}(\mathfrak{z};\mathfrak{Ne}_{\psi}^{-1}\mathfrak{N}'\mathfrak{e}_{\psi'}^{-1},\mathfrak{D})$ satisfies
\begin{align*}
G_{k,\psi}^{\psi'}(A\mathfrak{z};\mathfrak{Ne}_{\psi}^{-1}\mathfrak{N}'\mathfrak{e}_{\psi'}^{-1},\mathfrak{D})=(\psi\psi')(\gamma)\mathrm{N}(\gamma\mathfrak{z}+\delta)^{k}G_{k,\psi}^{\psi'}(\mathfrak{z};\mathfrak{Ne}_{\psi}^{-1}\mathfrak{N}'\mathfrak{e}_{\psi'}^{-1},\mathfrak{D})
\end{align*}
for $A=\left(\begin{smallmatrix}\alpha&\beta\\\gamma&\delta\end{smallmatrix}\right)\in\Gamma_{0}(\mathfrak{D}^{-1}\mathfrak{d}_{K}^{-1},\mathfrak{NN}'\mathfrak{D}\mathfrak{d}_{K})$ (\cite{Tsuyumine5}). We denote by \linebreak[4]$\mathbf{M}_{k}(\Gamma_{0}(\mathfrak{D}^{-1}\mathfrak{d}_{K}^{-1},\mathfrak{NN}'\mathfrak{D}\mathfrak{d}_{K}),\psi\psi')$, the space of Hilbert modular forms of weight $k$ with character $\psi\psi'$. For $f$ in the space, the value $\kappa(\alpha/\gamma,f)=\kappa(\alpha,\gamma,f)$ of $f(\mathfrak{z})$ at a cusp $\alpha/\gamma\ (\alpha\in\mathcal{O}_{K},\gamma\in\mathfrak{D}\mathfrak{d}_{K})$ is defined by $\kappa(\alpha/\gamma,f)=\kappa(\alpha,\gamma,f):=\lim_{\mathfrak{z}\to\sqrt{-1}\infty}\mathrm{N}(\gamma\mathfrak{z}+\delta)^{-k}f(B\mathfrak{z})$ where $B=\left(\begin{smallmatrix}\alpha&*\\\gamma&\delta\end{smallmatrix}\right)\in\mathrm{SL}_{2}(K)$. The subspace of $\mathbf{M}_{k}(\Gamma_{0}(\mathfrak{D}^{-1}\mathfrak{d}_{K}^{-1},\mathfrak{NN}'\mathfrak{D}\mathfrak{d}_{K}),\psi\psi')$ consisting of cusps forms is denoted by\linebreak$\mathbf{S}_{k}(\Gamma_{0}(\mathfrak{D}^{-1}\mathfrak{d}_{K}^{-1},\mathfrak{NN}'\mathfrak{D}\mathfrak{d}_{K}),\psi\psi')$, where $\psi\psi'$ is omitted if $\psi\psi'=\mathbf{1}_{\mathfrak{NN}'}$.

The values of Eisenstein series (\ref{eqn:e-s}) at cusps are computed in \cite{Tsuyumine5}. We can take $\alpha,\gamma$ so that $\mathfrak{B}:=(\alpha,\gamma\mathfrak{D}^{-1}\mathfrak{d}_{K}^{-1})$ is coprime to $\mathfrak{NN}'$. The value $\kappa(\alpha,\gamma,G_{k,\psi}^{\psi'}(\mathfrak{z};\mathfrak{Ne}_{\psi}^{-1}\mathfrak{N}'\mathfrak{e}_{\psi'}^{-1},\mathfrak{D}))$ at the cusp $\alpha/\gamma$ is $0$ if there are not integral ideals $\mathfrak{M}_{\gamma},\mathfrak{M}_{\gamma}'$ with $\mathfrak{M}_{\gamma}|\mathfrak{N},\,(\mathfrak{M}_{\gamma},\mathfrak{f}_{\psi})=\mathcal{O}_{K},\,\mathfrak{M}_{\gamma}'|\mathfrak{N}'\mathfrak{e}_{\psi'}^{-1}$ and with $(\gamma\mathfrak{D}^{-1}\mathfrak{d}_{K}^{-1},\mathfrak{N}\mathfrak{M}_{\gamma}^{-1}\mathfrak{N}')=\mathfrak{N}\mathfrak{M}_{\gamma}^{-1}\mathfrak{N}'\mathfrak{e}_{\psi'}^{-1}\mathfrak{M}_{\gamma}'^{-1}$. Suppose otherwise, and let $\mathfrak{M}_{\gamma}$ be the largest such ideal. Then the value is given by
\begin{align}
\label{mvc1}&\mathrm{sgn}^{\mathbf{e}_{\psi}}(\alpha)\mathrm{sgn}^{\mathbf{e}_{\psi'}}(-\gamma)\mu_{K}((\mathfrak{e}_{\psi}\mathfrak{f}_{\psi}^{-1},\mathfrak{M}_{\gamma}\mathfrak{N}'))\overline{\widetilde{\psi}}(\alpha\cdot\mathfrak{B}^{-1}\mathfrak{M}_{\gamma}\mathfrak{M}_{\gamma}'(\mathfrak{e}_{\psi}\mathfrak{f}_{\psi}^{-1},\mathfrak{M}_{\gamma}\mathfrak{N}')^{-1})\\
&\times\psi'(-\gamma\cdot\mathfrak{D}^{-1}\mathfrak{d}_{K}^{-1}\mathfrak{B}^{-1}\mathfrak{N}^{-1}\mathfrak{M}_{\gamma}\mathfrak{e}_{\psi}\mathfrak{f}_{\psi}^{-1}(\mathfrak{e}_{\psi}\mathfrak{f}_{\psi}^{-1},\mathfrak{M}_{\gamma}\mathfrak{N}')^{-1}\mathfrak{N}'^{-1}\mathfrak{e}_{\psi'}\mathfrak{M}_{\gamma}')\mathrm{N}(\mathfrak{B})^{k}\nonumber\\
&\times\mathrm{N}(\mathfrak{M}_{\gamma}^{-1}(\mathfrak{e}_{\psi}\mathfrak{f}_{\psi}^{-1},\mathfrak{M}_{\gamma}\mathfrak{N}')\mathfrak{f}_{\psi}\mathfrak{f}_{\overline{\psi}\psi'}^{-1})^{k-1}\mathrm{N}(\mathfrak{M}_{\gamma}')^{-k}\mathrm{N}(\mathfrak{f}_{\psi}\mathfrak{f}_{\overline{\psi}\psi'}^{-1})\tau_{K}(\overline{\widetilde{\psi}})^{-1}\tau_{K}(\widetilde{\overline{\psi}\psi'})\mathrm{N}(\mathfrak{M}_{\gamma}^{-1}) \nonumber\\
&\times L_{K}(1-k,\widetilde{\psi\overline{\psi}'})\prod_{\mathfrak{P}|\mathfrak{M}_{\gamma}}(1-\mathrm{N}(\mathfrak{P}))\ \prod_{\mathfrak{P}|\mathfrak{e}_{\psi'},\mathfrak{P}\nmid\mathfrak{f}_{\overline{\psi}\psi'}}(1-\widetilde{\overline{\psi}\psi'}(\mathfrak{P})\mathrm{N}(\mathfrak{P})^{-k})\nonumber\\
&\times\prod_{\mathfrak{P}|\mathfrak{e}_{\psi}\mathfrak{f}_{\psi}^{-1},\mathfrak{P}{\nmid}\mathfrak{M}_{\gamma}\mathfrak{N}'}(1-\widetilde{\psi\overline{\psi}'}(\mathfrak{P})\mathrm{N}(\mathfrak{P})^{k-1})\nonumber
\end{align}
where if $\gamma=0$, then the value is non-zero only when $\mathfrak{N}'=\mathcal{O}_{K}$ and it is given by replacing $\gamma$ in (\ref{mvc1}) by $\mathrm{N}(\mathfrak{N})$, and where if $\alpha=0$, the value is non-zero only when $\mathfrak{f}_{\psi}=\mathcal{O}_{K}$ and it is given by replacing $\alpha$ in (\ref{mvc1}) by $1$.

When $k=1$, the values of the Eisenstein series at cusps may may have an additional term. Let $\mathfrak{L}_{\gamma}:=\gamma\mathfrak{D}^{-1}\mathfrak{d}_{K}^{-1}\mathfrak{N}^{-1}\mathfrak{e}_{\psi}\mathfrak{N}'^{-1}\mathfrak{e}_{\psi'}\mathfrak{f}_{\psi'}^{-1}$ and $\mathfrak{L}_{\gamma}':=(\gamma\mathfrak{D}^{-1}\mathfrak{d}_{K}^{-1}\mathfrak{N}'^{-1},\mathfrak{N}^{-1})\cap\mathfrak{e}_{\psi'}^{-1}\mathfrak{f}_{\psi'}$. If there is an integral divisor $\mathfrak{R}$ of $\mathfrak{e}_{\psi'}\mathfrak{f}_{\psi'}^{-1}$ so that the numerator of $\mathfrak{L}_{\gamma}\mathfrak{R}^{-1}$ is coprime to $\mathfrak{N}$ and the denominator is coprime to $\mathfrak{f}_{\psi'}\mathfrak{R}$, then there is the additional term. Let $\widetilde{\mathfrak{R}}_{\gamma}$ be the divisor of $(\mathfrak{N},\mathfrak{e}_{\psi'}\mathfrak{f}_{\psi'}^{-1})$ satisfying $v_{\mathfrak{P}}(\mathfrak{L}_{\gamma}\widetilde{\mathfrak{R}}_{\gamma}^{-1})=0$ for any prime divisor $\mathfrak{P}$ of $(\mathfrak{N},\mathfrak{e}_{\psi'}\mathfrak{f}_{\psi'}^{-1})$. Then $\kappa(\alpha,\gamma,G_{k,\psi}^{\psi'}(\mathfrak{z};\mathfrak{Ne}_{\psi}^{-1}\mathfrak{N}'\mathfrak{e}_{\psi'}^{-1},\mathfrak{D}))$ has the additional term
\begin{align}
&\label{mvc2}\mathrm{sgn}^{\mathbf{e}_{\psi}}(-\gamma)\mathrm{sgn}^{\mathbf{e}_{\psi'}}(\alpha)\mu_{K}(\widetilde{\mathfrak{R}}_{\gamma})\psi(-\gamma\cdot\gamma^{-1}((\mathfrak{L}_{\gamma}\widetilde{\mathfrak{R}}_{\gamma}^{-1})\cap\mathcal{O}_{K}))\overline{\widetilde{\psi}'}(\alpha\cdot\mathfrak{B}^{-1})\widetilde{\psi}'(\widetilde{\mathfrak{R}}_{\gamma})\\
&\times\overline{\psi}_{\widetilde{\mathfrak{R}}_{\gamma}}'((\mathfrak{L}_{\gamma}\widetilde{\mathfrak{R}}_{\gamma}^{-1},\mathcal{O}_{K})^{-1})\overline{\psi}(\mathfrak{B})\mathrm{N}(\mathfrak{B})\varphi_{K}(\widetilde{\mathfrak{R}}_{\gamma}^{-1}\mathfrak{L}_{\gamma}'^{-1})\mathrm{N}((\mathfrak{L}_{\gamma},\widetilde{\mathfrak{R}}_{\gamma})\mathfrak{L}_{\gamma}')\mathrm{N}(\mathfrak{f}_{\psi'}\mathfrak{f}_{\psi\overline{\psi}'}^{-1})\tau_{K}(\overline{\widetilde{\psi}'})^{-1}\nonumber\\
&\times\tau_{K}(\widetilde{\psi\overline{\psi'}}) L_{K}(0,\widetilde{\overline{\psi}\psi'})\prod_{\mathfrak{P}|\mathfrak{e}_{\psi},\mathfrak{P}{\nmid}\mathfrak{f}_{\psi\overline{\psi}'}}(1-\widetilde{\psi\overline{\psi'}}(\mathfrak{P})\mathrm{N}(\mathfrak{P})^{-1})\prod_{\mathfrak{P}|\mathfrak{e}_{\psi'}\mathfrak{f}_{\psi'}^{-1}\mathfrak{L}_{\gamma}'}(1-\widetilde{\overline{\psi}\psi'}(\mathfrak{P}))\nonumber
\end{align}
where if $\gamma=0$, then the value is non-zero only when $\mathfrak{N}=\mathcal{O}_{K}$ and it is given by replacing $\gamma$ in (\ref{mvc2}) by $\mathrm{N}(\mathfrak{N}')$, and where if $\alpha=0$, the value is non-zero only when $\mathfrak{f}_{\psi'}=\mathcal{O}_{K}$ and it is given by replacing $\alpha$ in (\ref{mvc2}) by $1$.

Let $\theta(\mathfrak{z})=\sum_{\mu\in\mathcal{O}_{K}}\mathbf{e}(\mathrm{tr}(\mu^{2}\mathfrak{z}))$ be a theta series. It is a modular form of weight $1/2$ for $\Gamma_{0}(\mathfrak{d}_{K}^{-1},4\mathfrak{d}_{K})$. Let $j$ be its factor of automorphy, namely, $\theta(A\mathfrak{z})=j(A,\mathfrak{z})\theta(\mathfrak{z})$ for $A\in\Gamma_{0}(\mathfrak{d}_{K}^{-1},4\mathfrak{d}_{K})$. Suppose that $4|\mathfrak{N}$. Let $\psi_{0}\in C_{\mathfrak{N}}^{\ast}$ and let $k$ be a natural number with same parity as $\psi_{0}$. We denote by $\mathbf{M}_{k+1/2}(\Gamma_{0}(\mathfrak{d}_{K}^{-1},\mathfrak{N}\mathfrak{d}_{K}),\psi_{0})$, the space of modular forms with $\psi_{0}(\gamma)j(A,\mathfrak{z})\mathrm{N}(\gamma\mathfrak{z}+\delta)^{k}\ (A=\left(\begin{smallmatrix}\alpha&\beta\\\gamma&\delta\end{smallmatrix}\right)\in\Gamma_{0}(\mathfrak{d}_{K}^{-1},\mathfrak{N}\mathfrak{d}_{K}))$ as a factor of automorphy. Assume that the class number of $K$ is $1$ in the wide sense. Let $\alpha$ be a totally positive square-free integer in $K$. Then the Shimura lifting map $\mathscr{S}_{\alpha,\psi_{0}}$ associated with $\alpha$ should be a linear map of $\mathbf{M}_{k+1/2}(\Gamma_{0}(\mathfrak{d}_{K}^{-1},\mathfrak{N}\mathfrak{d}_{K}),\psi_{0})$ to $\mathbf{M}_{2k}(\Gamma_{0}(\mathfrak{d}_{K}^{-1},2^{-1}\mathfrak{N}\mathfrak{d}_{K}),\psi_{0}^{2})$ satisfying
\begin{align}
\mathscr{S}_{\alpha,\psi_{0}}(f)=C_{f,\alpha}+\sum_{\nu\in\mathcal{O}_{K},\succ0}\sum_{\delta|\nu,\delta\in\mathcal{O}_{K}/\mathcal{O}_{K}^{\times}}(\psi_{\alpha}\psi_{0})(\delta)\mathrm{N}((\delta))^{k-1}c_{\alpha\nu^{2}\delta^{-2}}\mathbf{e}(\mathrm{tr}(\nu\mathfrak{z}))\label{eqn:s-l-h}
\end{align}
with $f(\mathfrak{z})=c_{0}+\sum_{\nu\in\mathcal{O}_{K},\succ0}c_{\nu}\mathbf{e}(\mathrm{tr}(\nu\mathfrak{z}))\in\mathbf{M}_{k+1/2}(\Gamma_{0}(\mathfrak{d}_{K}^{-1},\mathfrak{N}\mathfrak{d}_{K}),\psi_{0})$ where $C_{f,\alpha}$ is a constant for which the left hand side of (\ref{eqn:s-l-h}) is a Hilbert modular form. As for a space of Hilbert cusp forms, Shimura \cite{Shimura} established such lifting map.  However the existence of the map for non-cuspidal Hilbert modular forms is not yet proved unconditionally. We show in \cite{Tsuyumine4}, for example that if $k\ge2$ and $16|\mathfrak{N}$, then there is a map $\mathscr{S}_{\alpha,\psi_{0}}$ of $\mathbf{M}_{k+1/2}(\Gamma_{0}(\mathfrak{d}_{K}^{-1},\mathfrak{N}\mathfrak{d}_{K}),\psi_{0})$, and further that there is a map of the subspace of $\mathbf{M}_{k+1/2}(\Gamma_{0}(\mathfrak{d}_{K}^{-1},\mathfrak{N}\mathfrak{d}_{K}),\psi_{0})$ with $4|mathfrak{N}$, generated by the products of the theta series and Eisenstein series (\ref{eqn:e-s}) with $\psi_{0}=\psi\psi'$ for even or odd Hecke characters $\psi,\psi'$.  The latter assertion can be slightly generalized to $\psi,\psi'$ as in (\ref{cond:psi}) because the key proposition (Proposition 6.5 in \cite{Tsuyumine4}) holds for such characters, where the proof is made by word-to-word translation of the original proof. 

Let $f(\mathfrak{z})=c_{0}+\sum_{\nu\in\mathcal{O}_{K},\succ0}c_{\nu}\mathbf{e}(\mathrm{tr}(\nu\mathfrak{z}))\in\mathbf{M}_{k}(\Gamma_{0}(\mathfrak{d}_{K}^{-1},\mathfrak{N}\mathfrak{d}_{K}),\psi_{0})$, and let $\mu\in\mathcal{O}_{K},\succ0$ be so that $\mu\mathfrak{e}_{\psi_{0}}|\mathfrak{N}$ and $\mu^{2}|\mathfrak{N}$. Then we define
\begin{align*}
(U(\mu)f)(\mathfrak{z}):=c_{0}+\sum_{\nu\in\mathcal{O}_{K},\succ0}c_{\mu\nu}\mathbf{e}(\mathrm{tr}(\nu\mathfrak{z})),
\end{align*}
which is in $\mathbf{M}_{k}(\Gamma_{0}(\mathfrak{d}_{K}^{-1},\mu^{-1}\mathfrak{N}\mathfrak{d}_{K}),\psi_{0})$.

\section{Hilbert modular forms on $\mathbf{Q}(\sqrt{3})$}
Let $K=\mathbf{Q}(\sqrt{3})$. Then $\mathfrak{d}_{K}=(2\sqrt{3})$, and the class number of $K$ in the wide sense is $1$. Let $\varepsilon_{0}:=2+\sqrt{3}$ be a fundamental unit, which has a positive norm. Put $\pi_{2}:=1+\sqrt{3}$ and $\mathfrak{p}_{2}=(\pi_{2})$. Then the ideal $(2)$ is decomposed as $(2)=\mathfrak{p}_{2}^{2}$. Let $\chi_{-4}^{K}:=\chi_{-4}\circ\mathrm{N}\in C_{\mathfrak{p}_{2}}^{\ast}$. For $\mu\in\mathcal{O}_{K}$, $\chi_{-4}^{K}(\mu)=\mathrm{sgn}(\mathrm{N}(\mu))$ for $\pi_{2}\nmid\mu$, and $\chi_{-4}^{K}(\mu)=0$ for $\pi_{2}|\mu$. The Hecke character $\chi_{-4}^{K}$ is odd. The conductor $\mathfrak{f}_{\chi_{-4}^{K}}$ of $\chi_{-4}^{K}$ is $\mathcal{O}_{K}$, and $\widetilde{\chi}_{-4}^{K}(\mu)=\mathrm{sgn}(\mathrm{N}(\mu))$, in particular $\widetilde{\chi}_{-4}^{K}(\mathfrak{p}_{2})=\mathrm{sgn}(\mathrm{N}(\pi_{2}))=-1$. We have $\tau_{K}(\widetilde{\chi}_{-4}^{K})=\widetilde{\chi}_{-4}^{K}(2\sqrt{3})=-1$. Let $\rho_{2}$ be the unique nontrivial character of $(\mathcal{O}_{K}/\mathfrak{p}_{2}^{2})^{\times}$, where $\rho_{2}(a+b\sqrt{3})=(-1)^{b}\ (a,b\in\mathbf{Z})$ for $(a+b\sqrt{3},\pi_{2})=1$. This $\rho_{2}$ is not an ideal class character.  Taking $\pi_{2}$ as $\varpi_{\mathfrak{p}_{2}}$ in Section \ref{sect:n-r}, put $\psi:=\rho_{2}\chi_{-4}^{K},\psi':=\rho_{2}\mathbf{1}_{\mathfrak{p}_{2}}\in(\mathcal{O}_{K}/\mathfrak{p}_{2}^{2})^{\ast}\times C_{\mathfrak{p}_{2}^{2}}^{\ast}$ where we regard $\chi_{-4}^{K}$ and $\mathbf{1}_{\mathfrak{p}_{2}}$ as elements of $C_{\mathfrak{p}_{2}^{2}}^{\ast}$. Then $\psi$ and $\psi'$ satisfy the condition (\ref{cond:psi}), and $\psi\psi'$ is equal to a Hecke character $\chi_{-4}^{K}$, and $\mathfrak{f}_{\psi}=\mathfrak{f}_{\psi'}=2,\tau_{K}(\psi)=2,\tau_{K}(\psi')=-2$.

The square $\theta(\mathfrak{z})^{2}$ of the theta series is a Hilbert modular form for $\Gamma_{0}(\mathfrak{d}_{K}^{-1},4\mathfrak{d}_{K})$ of weight $1$ with character $\chi_{-4}^{K}$. We take as a set $\mathcal{C}_{0}(4)$ of representatives of cusps of $\Gamma_{0}(\mathfrak{d}_{K}^{-1},4\mathfrak{d}_{K})$,
\begin{align*}
\mathcal{C}_{0}(4)=\{\tfrac{1}{8\sqrt{3}},\tfrac{1}{4\pi_{2}\sqrt{3}},\tfrac{1}{4\sqrt{3}},\tfrac{\varepsilon_{0}}{4\sqrt{3}},\tfrac{1}{2\pi_{2}\sqrt{3}},\tfrac{1}{2\sqrt{3}}\},
\end{align*}
and as a set $\mathcal{C}_{0}(2)$ of representatives of cusps of $\Gamma_{0}(\mathfrak{d}_{K}^{-1},2\mathfrak{d}_{K})$,
\begin{align*}
\mathcal{C}_{0}(2)=\{\tfrac{1}{4\sqrt{3}},\tfrac{1}{2\pi_{2}\sqrt{3}},\tfrac{1}{2\sqrt{3}}\}.
\end{align*}
Transformation formulas for theta series (see for example \cite{Tsuyumine4}, Sect.3) give the values of $\theta(\mathfrak{z})^{2}$ at cusps as in Table \ref{tbl:vtac}.
\begin{table}[htb]
\caption{Values of $\theta(\mathfrak{z})^{2}$ at cusps.}
$
\begin{array}{c|cccccc}
\mathcal{C}_{0}(4)&\tfrac{1}{8\sqrt{3}}&\tfrac{1}{4\pi_{2}\sqrt{3}}&\tfrac{1}{4\sqrt{3}}&\tfrac{\varepsilon_{0}}{4\sqrt{3}}&\tfrac{1}{2\pi_{2}\sqrt{3}}&\tfrac{1}{2\sqrt{3}}\\\hline
\theta(\mathfrak{z})^{2}&1&0&1&0&0&2^{-2}
\end{array}
$\label{tbl:vtac}
\end{table}

We express $\theta(\mathfrak{z})^{2}$ as a linear combination of Eisenstein series of weight $1$.

By (\ref{mvc1}) and (\ref{mvc2}), and by equations $L_{K}(0,\chi_{-4}^{K})=3^{-1}$ and $L_{K}(0,\widetilde{\chi}_{-4}^{K})=2^{-1}3^{-1}$, the values at cusps of Eisenstein series  $G_{1,\chi_{-4}^{K}}(\mathfrak{z};\mathcal{O}_{K},\mathcal{O}_{K}),G_{1,\chi_{-4}^{K}}^{\mathbf{1}_{\mathfrak{p}_{2}}}(\mathfrak{z};\mathcal{O}_{K},\mathcal{O}_{K})$,\linebreak $G_{1,\chi_{-4}^{K}}^{\mathbf{1}_{\mathfrak{p}_{2}}}(\mathfrak{z};(2),\mathcal{O}_{K}),G_{1,\psi}^{\psi'}(\mathfrak{z};\mathcal{O}_{K},\mathcal{O}_{K})$ are obtained as in Table \ref{tbl:ve1ac}, where these Eisenstein series are modular forms for $\Gamma_{0}(\mathfrak{d}_{K}^{-1},4\mathfrak{d}_{K})$ of weight $1$ with character $\chi_{-4}^{K}$.
\begin{table}[htb]
\caption{Values at cusps, of Eisenstein series of weight $1$.}
$
\begin{array}{c|cccccc}
\mathcal{C}_{0}(2)&\tfrac{1}{8\sqrt{3}}&\tfrac{1}{4\pi_{2}\sqrt{3}}&\tfrac{1}{4\sqrt{3}}&\tfrac{\varepsilon_{0}}{4\sqrt{3}}&\tfrac{1}{2\pi_{2}\sqrt{3}}&\tfrac{1}{2\sqrt{3}}\\\hline
G_{1,\chi_{-4}^{K}}(\mathfrak{z};\mathcal{O}_{K},\mathcal{O}_{K})&3^{-1}&3^{-1}&3^{-1}&3^{-1}&3^{-1}&3^{-1}\\
G_{1,\chi_{-4}^{K}}^{\mathbf{1}_{\mathfrak{p}_{2}}}(\mathfrak{z};\mathcal{O}_{K},\mathcal{O}_{K})&0&0&0&0&2^{-1}&2^{-2}\\
G_{1,\chi_{-4}^{K}}^{\mathbf{1}_{\mathfrak{p}_{2}}}(\mathfrak{z};(2),\mathcal{O}_{K})&0&2^{-1}&2^{-2}&2^{-2}&-2^{-3}&2^{-4}\\
G_{1,\psi}^{\psi'}(\mathfrak{z};\mathcal{O}_{K},\mathcal{O}_{K})&0&0&1&-1&0&0\\
\end{array}$
\label{tbl:ve1ac}
\end{table}

The product of the second and fourth Eisenstein series of Table \ref{tbl:ve1ac} is a cusp form for $\Gamma_{0}(\mathfrak{d}_{K}^{-1},4\mathfrak{d}_{K})$ of weight $2$ since it vanishes at all the cusps. The Fourier expansions of it and of its $U(2)$-image are as follows;
\begin{align}
\label{eqn:pefe}&2^{-4}G_{1,\chi_{-4}^{K}}^{\mathbf{1}_{\mathfrak{p}_{2}}}(\mathfrak{z};\mathcal{O}_{K},\mathcal{O}_{K})G_{1,\psi}^{\psi'}(\mathfrak{z};\mathcal{O}_{K},\mathcal{O}_{K})\\
=&\mathbf{e}(\mathrm{tr}(2\mathfrak{z}))-\mathbf{e}(\mathrm{tr}((4{-}2\sqrt{3})\mathfrak{z}))-\mathbf{e}(\mathrm{tr}((4{+}2\sqrt{3})\mathfrak{z}))-3\mathbf{e}(\mathrm{tr}(6\mathfrak{z}))\nonumber\\
&+2\mathbf{e}(\mathrm{tr}((8{-}2\sqrt{3})\mathfrak{z}))+2\mathbf{e}(\mathrm{tr}((8{+}2\sqrt{3})\mathfrak{z}))+\cdots,\nonumber\\
\label{eqn:Xi}&\Xi(\mathfrak{z}):=U(2)(2^{-4}G_{1,\chi_{-4}^{K}}^{\mathbf{1}_{\mathfrak{p}_{2}}}(\mathfrak{z};\mathcal{O}_{K},\mathcal{O}_{K})G_{1,\psi}^{\psi'}(\mathfrak{z};\mathcal{O}_{K},\mathcal{O}_{K}))\\
=&\mathbf{e}(\mathrm{tr}(\mathfrak{z}))-\mathbf{e}(\mathrm{tr}((2{-}\sqrt{3})\mathfrak{z}))-\mathbf{e}(\mathrm{tr}((2{+}\sqrt{3})\mathfrak{z}))-3\mathbf{e}(\mathrm{tr}(3\mathfrak{z}))\nonumber\\
&+2\mathbf{e}(\mathrm{tr}((4{-}\sqrt{3})\mathfrak{z}))+2\mathbf{e}(\mathrm{tr}((4{+}\sqrt{3})\mathfrak{z}))+\cdots,\nonumber
\end{align}
where $\Xi(\mathfrak{z})$ is in $\mathbf{S}_{2}(\Gamma(\mathfrak{d}_{K}^{-1},2\mathfrak{d}_{K}))$.

\begin{lemma} The following equality holds;
\begin{align}
\label{eqn:squaretheta}\theta(\mathfrak{z})^{2}=&3G_{1,\chi_{-4}^{K}}(\mathfrak{z};\mathcal{O}_{K},\mathcal{O}_{K})-2^{-1}5G_{1,\chi_{-4}^{K}}^{\mathbf{1}_{\mathfrak{p}_{2}}}(\mathfrak{z};\mathcal{O}_{K},\mathcal{O}_{K})\\
&-2G_{1,\chi_{-4}^{K}}^{\mathbf{1}_{\mathfrak{p}_{2}}}(\mathfrak{z};(2),\mathcal{O}_{K})+2^{-1}G_{1,\psi}^{\psi'}(\mathfrak{z};\mathcal{O}_{K},\mathcal{O}_{K}).\nonumber
\end{align}
\end{lemma}
\begin{proof} Let $\Gamma(\mathfrak{d}_{K}^{-1},\mathfrak{d}_{K})[2]:=\{\left(\begin{smallmatrix}\alpha&\beta\\\gamma&\delta\end{smallmatrix}\right)\in\mathrm{SL}_{2}(K)\mid\alpha\equiv\delta\equiv1\pmod{2},\beta\in 2\mathfrak{d}_{K}^{-1},\gamma\in 2\mathfrak{d}_{K}\}$, which is called a congruence subgroup of level $2$. Then $\Gamma(\mathfrak{d}_{K}^{-1},\mathfrak{d}_{K})[2]$ acts freely on $\mathfrak{H}^{2}$. The arithmetic genus of a compactified nonsingular model of\linebreak $\Gamma(\mathfrak{d}_{K}^{-1},\mathfrak{d}_{K})[2]\backslash\mathfrak{H}^{2}$ is determined by the volume and by the contribution from cusp singularities (van der Geer \cite{Geer} Chap.\,II$\sim$IV), and it is computed to be $4+1/12\,(4+4+4+4+4+4)=6$. Hence $\dim \mathbf{S}_{2}(\Gamma(\mathfrak{d}_{K}^{-1},\mathfrak{d}_{K})[2])=5$ where $\mathbf{S}_{2}(\Gamma(\mathfrak{d}_{K}^{-1},\mathfrak{d}_{K})[2])$ denotes the space of cusp forms for $\Gamma(\mathfrak{d}_{K}^{-1},\mathfrak{d}_{K})[2]$ of weight $2$. Since $\left(\begin{smallmatrix}2&0\\0&1\end{smallmatrix}\right)^{-1}\Gamma(\mathfrak{d}_{K}^{-1},\mathfrak{d}_{K})[2]\left(\begin{smallmatrix}2&0\\0&1\end{smallmatrix}\right)$  is a subgroup of $\Gamma(\mathfrak{d}_{K}^{-1},4\mathfrak{d}_{K})$ of index $2$, the dimension of the space of cusp forms for $\Gamma(\mathfrak{d}_{K}^{-1},4\mathfrak{d}_{K})$ of weight $2$ is at most $5$.

Let $f(\mathfrak{z})$ is a Hilbert modular form given as the left hand side minus the right hand side of the equation (\ref{eqn:squaretheta}). Then it is a cusp form from Table \ref{tbl:vtac} and Table \ref{tbl:ve1ac}. It is easy to check that the Fourier coefficients of $f$ are $0$ at least for $\nu=1,2,2\pm\sqrt{3}$. We show that $f=0$. Suppose that  $f\ne0$. Then  $f(\mathfrak{z})G_{1,\chi_{-4}^{K}}(\mathfrak{z};\mathcal{O}_{K},\mathcal{O}_{K}),f(\mathfrak{z})G_{1,\chi_{-4}^{K}}^{\mathbf{1}_{\mathfrak{p}_{2}}}(\mathfrak{z};\mathcal{O}_{K},\mathcal{O}_{K})$, $f(\mathfrak{z})G_{1,\chi_{-4}^{K}}^{\mathbf{1}_{\mathfrak{p}_{2}}}(\mathfrak{z};(2),\mathcal{O}_{K})),f(\mathfrak{z})G_{1,\psi}^{\psi'}(\mathfrak{z};\mathcal{O}_{K},\mathcal{O}_{K}),f(\mathfrak{z})^{2}$ are linearly independent.  The first four Fourier coefficients of these cusp forms vanish. Then these cusp forms together with $G_{1,\chi_{-4}^{K}}^{\mathbf{1}_{\mathfrak{p}_{2}}}(\mathfrak{z};\mathcal{O}_{K},\mathcal{O}_{K})G_{1,\psi}^{\psi'}(\mathfrak{z};\mathcal{O}_{K},\mathcal{O}_{K})$ and $\Xi(\mathfrak{z})$ are linearly independent by the Fourier expansions (\ref{eqn:pefe}) and (\ref{eqn:Xi}). This contradicts to $\dim \mathbf{S}_{2}(\Gamma(\mathfrak{d}_{K}^{-1},4\mathfrak{d}_{K}))\le 5$. Hence $f=0$.
\end{proof}

Comparing the Fourier coefficients of the both sides of (\ref{eqn:squaretheta}), we obtain the following;
\begin{corollary} Let $r_{2,K}(\nu)$ denote the number representing  $\nu$ as sums of two integer squares in $K$. Let $\sigma_{0,\psi}^{\psi'}$ be as in (\ref{def:sigma1}). Then 
\begin{align*}
r_{2,K}(\nu)=2\{6\sigma_{0,\chi_{-4}^{K}}(\nu)+(\rho_{2}(\nu)-5)\sigma_{0,\chi_{-4}^{K}}^{\mathbf{1}_{\mathfrak{p}_{2}}}(\nu)-4\sigma_{0,\chi_{-4}^{K}}^{\mathbf{1}_{\mathfrak{p}_{2}}}(\nu/2)\}\quad(\nu\in\mathcal{O}_{K},\succ0).
\end{align*}
\end{corollary}

In other words, for $\nu=a+b\sqrt{3}\succ0$, we have (i) $r_{2,K}(\nu)=0$ if  $v_{\mathfrak{p}_{2}}(\nu)=0$ and $b$ is odd,  or if $v_{\mathfrak{p}_{2}}(\nu)$ is odd, and (ii) $r_{2,K}(\nu)=4\sigma_{0,\chi_{-4}^{K}}(\nu)$ if $v_{\mathfrak{p}_{2}}(\nu)=0$ and $b$ is even, or if $v_{\mathfrak{p}_{2}}(\nu)=2$, and (iii) $r_{2,K}(\nu)=12\sigma_{0,\chi_{-4}^{K}}(\nu)$ if $2|v_{\mathfrak{p}_{2}}(\nu)\ge4$. 

The following corollary is for later use, which is rather technical.
\begin{corollary}\label{cor:stse}
(i) Assume that $\nu$ is odd, and it is represented as a sum of two integer squares. Then $x_{1},x_{2}\in\mathcal{O}_{K}$ satisfying $x_{1}^{2}+x_{2}^{2}=\pi_{2}^{2}\nu$, are both even.

(ii) Assume that $\nu\succ0,v_{\mathfrak{p}_{2}}(\nu)\ge4$, and $\nu$ is represented as a sum of two integer squares. Then there are $x_{1},x_{2}\in\mathfrak{p}_{2}$ satisfying $x_{1}^{2}+x_{2}^{2}=\nu$.
\end{corollary}
\begin{proof} (i) Since $\nu$ is represented as a sum of two squares, $\rho_{2}(\nu)$ is $1$. Then $r_{2,K}(\pi_{2}^{2}\nu)=4\sigma_{0,\chi_{-4}^{K}}(\pi_{2}^{2}\nu)=4\sigma_{0,\chi_{-4}^{K}}(\nu)=r_{2,K}(\nu)$. Hence all the solutions of $x_{1}^{2}+x_{2}^{2}=\pi_{2}^{2}\nu$ is obtained by $x_{1}=\pi_{2}x_{1}',x_{2}=\pi_{2}x_{2}'$ with $x_{1}'^{2}+x_{2}'^{2}=\nu$, which shows the assertion.

(ii) Since $\nu$ is represented as a sum of two squares, $v_{\mathfrak{p}_{2}}(\nu)$ is even. Then $r_{2,K}(\pi^{2-v_{\mathfrak{p}_{2}}(\nu)}\nu)>0$, For solutions $x_{1}',x_{2}'$ of $x_{1}'^{2}+x_{2}'^{2}=\pi^{2-v_{\mathfrak{p}_{2}}(\nu)}\nu$ , $x_{1}=\pi_{2}^{v_{\mathfrak{p}_{2}}(\nu)/2-1}x_{1}'$ and $x_{2}=\pi_{2}^{v_{\mathfrak{p}_{2}}(\nu)/2-1}x_{2}'$ satisfy $x_{1}^{2}+x_{2}^{2}=\nu$.
\end{proof}

We derive from Corollary \ref{cor:stse}, the following result on sums of two integer squares in $K$. Let $p_{i}>0$ denote an odd prime element of $O_{K}$ with positive norm, and let $q_{j}>0$ denote an odd prime element of $\mathcal{O}_{K}$ with negative norm. We may assume that the coefficient of $\sqrt{3}$ in $p_{i}$ is even by multiplying by the fundamental unit $\varepsilon_{0}$ if necessary. Then a totally positive integer $\nu\in\mathcal{O}_{K}$ has a prime factorization as $\nu=\varepsilon_{0}^{k}\pi_{2}^{e}p_{1}^{m_{1}}\cdots p_{s}^{m_{s}}q_{1}^{n_{1}}\cdots q_{t}^{n_{t}}$ with $e+\sum_{1\le j\le t} n_{j}\equiv0\pmod{2}$. Then the necessary and sufficient condition that $\nu$ is expressed as a sum of two integers in $K$, is that (I) $e=0$ and $k\equiv n_{1}\equiv\cdots n_{t}\equiv0\pmod{2}$, or that (II) $2|e>0$ and $n_{1}\equiv\cdots n_{t}\equiv0\pmod{2}$. The number of representations is $4\prod_{i}(1+m_{i})$ for $e=0,2$, and it is $12\prod_{i}(1+m_{i})$ for $e\ge4,e\equiv0\pmod{2}$.

\section{Quadratic extensions of $\mathbf{Q}(\sqrt{3})$}
Let $\alpha\in\mathcal{O}_{K}$ be square-free and not necessarily totally positive, and let $F=K(\sqrt{\alpha})$. We denote by $\psi_{\alpha},\mathfrak{f}_{\alpha}(=d_{F/K})$ and $\mathfrak{d}_{F/K}$, the character associated the extension, the conductor of the extension and the relative different of $F$ over $K$ respectively. The norm map of $F/K$ and the norm map of $F/\mathbf{Q}$ are denoted by $\mathrm{N}_{F/K}$ and $\mathrm{N}_{F}$ respectively. Let $\chi_{-4}^{F}:=\chi_{-4}\circ\mathrm{N}_{F}=\chi_{-4}^{K}\circ\mathrm{N}_{F/K}$. We classify quadratic extensions by the congruence condition on $\alpha$ as in Table \ref{tbl:qeK}, where we understand that $\alpha\equiv1\pmod{4\mathfrak{p}_{2}}$ or $\alpha\equiv3\pmod{4\mathfrak{p}_{2}}$ in Case $(\mathrm{A})$ and so on. 
\begin{table}[htb]
\caption{Quadratic extensions $F=K(\sqrt{3},\sqrt{\alpha})$ of $\mathbf{Q}(\sqrt{3})$.}
$
\begin{array}{ccccc}
&\alpha\equiv&v_{\mathfrak{p}_{2}}(\mathfrak{f}_{\alpha})&\mathfrak{d}_{F/K}&\mathfrak{p}_{2}\mbox{ at }F\\\hline
(\mathrm{A})&1,3\ (\bmod\,4\mathfrak{p}_{2})&0&(\sqrt{\alpha})&\mbox{split}\\
(\mathrm{B})&5,7\ (\bmod\,4\mathfrak{p}_{2})&0&(\sqrt{\alpha})&\mbox{inert}\\
(\mathrm{C}_{1})&1+2\sqrt{3},3+2\sqrt{3}\ (\bmod\,4)&2&(\pi_{2}\sqrt{\alpha})&\mbox{ramify}\\
(\mathrm{C}_{2})&\sqrt{3}\ (\bmod\,2)&4&(2\sqrt{\alpha})&\mbox{ramify}\\
(\mathrm{D})&1+\sqrt{3}\ (\bmod\,2)&5&(2\sqrt{\alpha})&\mbox{ramify}\\\hline
\end{array}$
\label{tbl:qeK}
\end{table}

Let $\mathfrak{P}_{2}\subset\mathcal{O}_{F}$ be the ideal with $\mathfrak{P}_{2}^{2}=\mathfrak{p}_{2}\mathcal{O}_{F}$ in Cases ($\mathrm{C}_{1}$),($\mathrm{C}_{2}$) and ($\mathrm{D}$). Then $\chi_{-4}^{F}\in C_{\mathfrak{p}_{2}\mathcal{O}_{F}}^{\ast}$ in Cases ($\mathrm{A}$) and ($\mathrm{B}$), and $\chi_{-4}^{F}\in C_{\mathfrak{P}_{2}}^{\ast}$ in Cases ($\mathrm{C}_{1}$),($\mathrm{C}_{2}$) and ($\mathrm{D}$). In either case, the conductor $\mathfrak{f}_{\chi_{-4}^{F}}$ is $\mathcal{O}_{F}$. There holds
\begin{align*}
\chi_{-4}^{K}\psi_{\alpha}=\begin{cases}\psi_{-\alpha}\mathbf{1}_{\mathfrak{p}_{2}}&((\mathrm{A}),(\mathrm{B})),\\
\psi_{-\alpha}&((\mathrm{C}_{1}),(\mathrm{C}_{2}),(\mathrm{D})).\end{cases}
\end{align*}
Let $\widetilde{\chi}_{-4}^{F}$ be the primitive character associated with $\chi_{-4}^{F}$. Then $\widetilde{\chi}_{-4}^{F}\in C_{\mathcal{O}_{F}}^{\ast}$, and $\tau_{F}(\widetilde{\chi}_{-4}^{F})=1$. In Cases ($\mathrm{C}_{1}$),($\mathrm{C}_{2}$) and ($\mathrm{D}$), the value of $\widetilde{\chi}_{-4}^{F}$ at $\mathfrak{P}_{2}$ is $-1$. There holds
\begin{align}
L_{F}(s,\chi_{-4}^{F})&=L_{K}(s,\chi_{-4}^{K})L_{K}(s,\psi_{-\alpha})(1-\psi_{-\alpha}(\mathfrak{p}_{2})2^{-s}),\nonumber\\
L_{F}(s,\widetilde{\chi}_{-4}^{F})&=L_{K}(s,\widetilde{\chi}_{-4}^{K})L_{K}(s,\psi_{-\alpha}),\label{eqn:vl}
\end{align}
and hence $L_{F}(0,\widetilde{\chi}_{-4}^{F})$ equals $2^{-1}3^{-1}L_{K}(0,\psi_{-\alpha})$, and $L_{F}(0,\chi_{-4}^{F})$ equals $2\cdot3^{-1}\times$ $L_{K}(0,\psi_{-\alpha})$ in Case $(\mathrm{A})$, $0$ in Case $(\mathrm{B})$, or $3^{-1}L_{K}(0,\psi_{-\alpha})$ in Case $(\mathrm{C}_{1})$, $(\mathrm{C}_{2})$ or  $(\mathrm{D})$. Let $\psi,\psi'$ be as in the precedent section. Then $\mathfrak{f}_{\psi\circ\mathrm{N}_{F/K}}=\mathfrak{f}_{\psi'\circ\mathrm{N}_{F/K}}=(2)$ and $\tau_{F}(\psi\circ\mathrm{N}_{F/K})=\tau_{F}(\psi'\circ\mathrm{N}_{F/K})=4$ in Cases $(\mathrm{A})$ and $(\mathrm{B})$, and $\mathfrak{f}_{\psi\circ\mathrm{N}_{F/K}}=\mathfrak{f}_{\psi'\circ\mathrm{N}_{F/K}}=\mathfrak{p}_{2}$ and $\tau_{F}(\psi\circ\mathrm{N}_{F/K})=\tau_{F}(\psi'\circ\mathrm{N}_{F/K})=2$ in Cases $(\mathrm{C}_{1})$,$(\mathrm{C}_{2})$ and $(\mathrm{D})$.

Let $\alpha$ be a totally positive square-free integer of $K$. By (\ref{eqn:squaretheta}), we have
\begin{align}
\label{eqn:sltt}&\mathscr{S}_{\alpha,\chi_{-4}^{K}}(\theta(\mathfrak{z})^{3})\\
=&3\mathscr{S}_{\alpha,\chi_{-4}^{K}}(\theta(\mathfrak{z})G_{1,\chi_{-4}^{K}}(\mathfrak{z};\mathcal{O}_{K},\mathcal{O}_{K}))-2^{-1}5\mathscr{S}_{\alpha,\chi_{-4}^{K}}(\theta(\mathfrak{z})G_{1,\chi_{-4}^{K}}^{\mathbf{1}_{\mathfrak{p}_{2}}}(\mathfrak{z};\mathcal{O}_{K},\mathcal{O}_{K}))\nonumber\\
&-2\mathscr{S}_{\alpha,\chi_{-4}^{K}}(\theta(\mathfrak{z})G_{1,\chi_{-4}^{K}}^{\mathbf{1}_{\mathfrak{p}_{2}}}(\mathfrak{z};(2),\mathcal{O}_{K}))+2^{-1}\mathscr{S}_{\alpha,\chi_{-4}^{K}}(\theta(\mathfrak{z})G_{1,\psi}^{\psi'}(\mathfrak{z};\mathcal{O}_{K},\mathcal{O}_{K})),\nonumber
\end{align}
which is a Hilbert modular form of weight $2$ for $\Gamma_{0}(\mathfrak{d}_{K}^{-1},2\mathfrak{d}_{K})$.

Shimura lifts of products of the theta series and Eisenstein series are explicitly constructed in \cite{Tsuyumine4} Sect. 6, Sect. 7, which are essentially the restricts to the diagonal, of Hilbert-Eisenstein series on the field $F$. Let $\iota:\mathfrak{H}^{2}\longrightarrow\mathfrak{H}^{4}$ be the diagonal map associated with the inclusion of $K$ into $F$. For a Hilbert modular form $f(\mathfrak{Z})\ (\mathfrak{Z}\in\mathfrak{H}^{4})$ on $F$ of weight $k$, $f(\iota(\mathfrak{z}))\ (\mathfrak{z}\in\mathfrak{H}^{2})$ is a Hilbert modular from on $K$ of weight $2k$. We put
\begin{align*}
\lambda_{2,\chi_{-4}^{K}}(\mathfrak{z};(\alpha),\mathcal{O}_{K})&:=\begin{cases}U(2)(G_{1,\chi_{-4}^{F}}(\iota(\mathfrak{z}),\mathcal{O}_{F},\mathfrak{d}_{F/K}^{-1}))&(\mathrm{A}),(\mathrm{B}),\\
U(2)(G_{1,\chi_{-4}^{F}}(\iota(\mathfrak{z}),\mathfrak{p}_{2}\mathcal{O}_{F},\mathfrak{d}_{F/K}^{-1}))&(\mathrm{C_{1}}),\\
G_{1,\chi_{-4}^{F}}(\iota(\mathfrak{z}),\mathcal{O}_{F},\mathfrak{d}_{F/K}^{-1})&(\mathrm{C_{2}}),(\mathrm{D}),
\end{cases}\\
\lambda_{2,\chi_{-4}^{K}}^{\mathbf{1}_{\mathfrak{p}_{2}}}(\mathfrak{z};(\alpha),\mathcal{O}_{K})&:=\begin{cases}U(2)(G_{1,\chi_{-4}^{F}}^{\mathbf{1}_{\mathfrak{p}_{2}\mathcal{O}_{F}}}(\iota(\mathfrak{z}),\mathcal{O}_{F},\mathfrak{d}_{F/K}^{-1}))&(\mathrm{A}),(\mathrm{B}),\\
U(2)(G_{1,\chi_{-4}^{F}}^{\mathbf{1}_{\mathfrak{P}_{2}}}(\iota(\mathfrak{z}),\mathfrak{p}_{2}\mathcal{O}_{F},\mathfrak{d}_{F/K}^{-1}))&(\mathrm{C_{1}}),\\
G_{1,\chi_{-4}^{F}}^{\mathbf{1}_{\mathfrak{P}_{2}}}(\iota(\mathfrak{z}),\mathcal{O}_{F},\mathfrak{d}_{F/K}^{-1})&(\mathrm{C_{2}}),(\mathrm{D}),
\end{cases}\\
\lambda_{2,\chi_{-4}^{K}}^{\mathbf{1}_{\mathfrak{p}_{2}}}(\mathfrak{z};(\alpha),(2))&:=\begin{cases}U(2)(G_{1,\chi_{-4}^{F}}^{\mathbf{1}_{\mathfrak{p}_{2}\mathcal{O}_{F}}}(\iota(\mathfrak{z}),\mathfrak{p}_{2}\mathcal{O}_{F},\mathfrak{d}_{F/K}^{-1})))&(\mathrm{A}),(\mathrm{B}),\\
G_{1,\chi_{-4}^{F}}^{\mathbf{1}_{\mathfrak{P}_{2}}}(\iota(\mathfrak{z}),\mathcal{O}_{F},\mathfrak{d}_{F/K}^{-1})&(\mathrm{C_{1}}),\\
G_{1,\chi_{-4}^{F}}^{\mathbf{1}_{\mathfrak{P}_{2}}}(\iota(\mathfrak{z}),\mathfrak{p}_{2}\mathcal{O}_{F},\mathfrak{d}_{F/K}^{-1})&(\mathrm{C_{2}}),(\mathrm{D}),
\end{cases}\\
\lambda_{2,\psi}^{\psi'}(\mathfrak{z};(\alpha),\mathcal{O}_{K})&:=\begin{cases}U(2)(G_{1,\psi\circ\mathrm{N}_{F/K}}^{\psi'\circ\mathrm{N}_{F/K}}(\iota(\mathfrak{z}),\mathcal{O}_{F},\mathfrak{d}_{F/K}^{-1}))&(\mathrm{A}),(\mathrm{B}),\\
U(2)(G_{1,\psi\circ\mathrm{N}_{F/K}}^{\psi'\circ\mathrm{N}_{F/K}}(\iota(\mathfrak{z}),\mathfrak{p}_{2}\mathcal{O}_{F},\mathfrak{d}_{F/K}^{-1}))&(\mathrm{C_{1}}),\\
G_{1,\psi\circ\mathrm{N}_{F/K}}^{\psi'\circ\mathrm{N}_{F/K}}(\iota(\mathfrak{z}),\mathcal{O}_{F},\mathfrak{d}_{F/K}^{-1})&(\mathrm{C_{2}}),(\mathrm{D}).
\end{cases}
\end{align*}
In \cite{Tsuyumine4}, $\lambda_{2,\chi_{-4}^{K}}(\mathfrak{z};(\alpha),\mathcal{O}_{K})$, $\lambda_{2,\chi_{-4}^{K}}^{\mathbf{1}_{\mathfrak{p}_{2}}}(\mathfrak{z};(\alpha),\mathcal{O}_{K}),\cdots,$ are denoted by \linebreak$\lambda_{2,\chi_{-4}^{K}}(\mathfrak{z};(\alpha),\mathcal{O}_{K},\mathcal{O}_{K},\mathcal{O}_{K})$, $\lambda_{2,\chi_{-4}^{K}}^{\mathbf{1}_{\mathfrak{p}_{2}}}(\mathfrak{z};(\alpha),\mathcal{O}_{K},\mathcal{O}_{K},\mathcal{O}_{K}),\cdots$. We drop the last two ideals from the notation $\lambda_{2,\chi_{-4}^{K}}(\mathfrak{z};(\alpha),\mathcal{O}_{K},\mathcal{O}_{K},\mathcal{O}_{K})$ and so on because they are always $\mathcal{O}_{K}$ in the present paper.

By \cite{Tsuyumine4}, we have $\mathscr{S}_{\alpha,\chi_{-4}^{K}}(\theta(\mathfrak{z})G_{1,\chi_{-4}^{K}}(\mathfrak{z};\mathcal{O}_{K},\mathcal{O}_{K}))=2^{-2}\lambda_{2,\chi_{-4}^{K}}(\mathfrak{z};(\alpha),\mathcal{O}_{K})$, \linebreak$\mathscr{S}_{\alpha,\chi_{-4}^{K}}(\theta(\mathfrak{z})G_{1,\chi_{-4}^{K}}^{\mathbf{1}_{\mathfrak{p}_{2}}}(\mathfrak{z};\mathcal{O}_{K},\mathcal{O}_{K}))=2^{-2}\lambda_{2,\chi_{-4}^{K}}^{\mathbf{1}_{\mathfrak{p}_{2}}}(\mathfrak{z};(\alpha),\mathcal{O}_{K})$, $\mathscr{S}_{\alpha,\chi_{-4}^{K}}(\theta(\mathfrak{z})G_{1,\chi_{-4}^{K}}^{\mathbf{1}_{\mathfrak{p}_{2}}}(\mathfrak{z};(2),$ $\mathcal{O}_{K}))=\lambda_{2,\chi_{-4}^{K}}^{\mathbf{1}_{\mathfrak{p}_{2}}}(\mathfrak{z};(\alpha),(2))$, and $\mathscr{S}_{\alpha,\chi_{-4}^{K}}(\theta(\mathfrak{z})G_{1,\psi}^{\psi'}(\mathfrak{z};\mathcal{O}_{K},\mathcal{O}_{K})=\lambda_{2,\psi}^{\psi'}(\mathfrak{z};(\alpha),\mathcal{O}_{K})$. We can obtain the values at cusps of Eisenstein series on $F$ by (\ref{mvc1}) and (\ref{mvc2}), which are all rational multiples of $L_{F}(0,\widetilde{\chi}_{-4}^{F})$, and hence we obtain the values at cusps of $\lambda_{2,\chi_{-4}^{K}}(\mathfrak{z};(\alpha),\mathcal{O}_{K})$, $\lambda_{2,\chi_{-4}^{K}}^{\mathbf{1}_{\mathfrak{p}_{2}}}(\mathfrak{z};(\alpha),\mathcal{O}_{K})$, $\lambda_{2,\psi}^{\psi'}(\mathfrak{z};(\alpha),\mathcal{O}_{K})$ and $\lambda_{2,\psi}^{\psi'}(\mathfrak{z};(\alpha),\mathcal{O}_{K})$ by using \cite{Tsuyumine4} Lemma 8.2. Then the equation (\ref{eqn:sltt}) gives the values at cusps of $\mathscr{S}_{\alpha,\chi_{-4}^{K}}(\theta(\mathfrak{z})^{3})$ as in Table \ref{tbl:vtacst}.
\begin{table}[htb]
\caption{Values of $L_{F}(0,\widetilde{\chi}_{-4}^{F})^{-1}\mathscr{S}_{\alpha,\chi_{-4}^{K}}(\theta(\mathfrak{z})^{3})$ at cusps in $\mathcal{C}_{0}(2)$.}
$
\begin{array}{c|cccl}
\mathcal{C}_{0}(2)&\tfrac{1}{4\sqrt{3}}&\tfrac{1}{2\pi_{2}\sqrt{3}}&\tfrac{1}{2\sqrt{3}}\\\hline
&3&3&2^{-2}3\cdot7&(\mathrm{A})\\
\smash{\raisebox{-.5em}{$\frac{\mathscr{S}_{\alpha,\chi_{-4}^{K}}(\theta(\mathfrak{z})^{3})}{L_{F}(0,\widetilde{\chi}_{-4}^{F})}$}}&0&0&2^{-2}3^{2}&(\mathrm{B})\\
&2^{-1}3&2^{-1}3&2^{-1}3&(\mathrm{C}_{1})\\
&2^{-1}3&2^{-1}3&2^{-3}3&(\mathrm{C}_{2}),(\mathrm{D})
\end{array}
$
\label{tbl:vtacst}
\end{table}

\section{Sums of three squares in $\mathbf{Q}(\sqrt{3})$}\label{sect:sts}
The Hilbert modular surface of the group $\Gamma_{0}(\mathfrak{d}_{K}^{-1},2\mathfrak{d}_{K})$ is a blown up K3 surface (see van der Geer \cite{Geer} Chap. VII), $\dim\mathbf{S}_{2}(\Gamma_{0}(\mathfrak{d}_{K}^{-1},2\mathfrak{d}_{K}))$ is one dimensional. The cusp form $\Xi(\mathfrak{z})$ of (\ref{eqn:Xi}) is a generator of the space $\mathbf{S}_{2}(\Gamma_{0}(\mathfrak{d}_{K}^{-1},2\mathfrak{d}_{K}))$. By (\ref{mvc1}) and (\ref{mvc2}), and by using $\zeta_{K}(-1)=1/6$ and $L(-1,\chi_{12})=-2$, the the values at cusps, of Eisenstein series $G_{2}(\mathfrak{z};\mathcal{O}_{K},\mathcal{O}_{K})$, $G_{2,\mathbf{1}_{\mathfrak{p}_{2}}}(\mathfrak{z};\mathcal{O}_{K},\mathcal{O}_{K})$, $G_{2,\mathbf{1}_{\mathfrak{p}_{2}}}(\mathfrak{z};\mathfrak{p}_{2},\mathcal{O}_{K})$ are obtained as in Table \ref{tbl:ve2ac}. From the table we see that these three Eisenstein series are linearly independent, and hence $\mathbf{M}_{2}(\Gamma_{0}(\mathfrak{d}_{K}^{-1},2\mathfrak{d}_{K}))=\langle G_{2}(\mathfrak{z};\mathcal{O}_{K},\mathcal{O}_{K}), G_{2,\mathbf{1}_{\mathfrak{p}_{2}}}(\mathfrak{z};\mathcal{O}_{K},\mathcal{O}_{K}), \linebreak G_{2,\mathbf{1}_{\mathfrak{p}_{2}}}(\mathfrak{z};\mathfrak{p}_{2},\mathcal{O}_{K}), \Xi(\mathfrak{z})\rangle$. 
\begin{table}[htb]
\caption{Values at cusps, of Eisenstein series of weight $2$.}
$
\begin{array}{c|ccc}
\mathcal{C}_{0}(2)&\tfrac{1}{4\sqrt{3}}&\tfrac{1}{2\pi_{2}\sqrt{3}}&\tfrac{1}{2\sqrt{3}}\\\hline
G_{2}(\mathfrak{z};\mathcal{O}_{K},\mathcal{O}_{K})&2^{-1}3^{-1}&2^{-1}3^{-1}&2^{-1}3^{-1}\\
G_{2,\mathbf{1}_{\mathfrak{p}_{2}}}(\mathfrak{z};\mathcal{O}_{K},\mathcal{O}_{K})&-2^{-1}3^{-1}&-2^{-1}3^{-1}&2^{-2}3^{-1}\\
G_{2,\mathbf{1}_{\mathfrak{p}_{2}}}(\mathfrak{z};\mathfrak{p}_{2},\mathcal{O}_{K})&-2^{-1}3^{-1}&2^{-2}3^{-1}&2^{-4}3^{-1}
\end{array}
$
\label{tbl:ve2ac}
\end{table}

Let $r_{3,K}(\nu)$ denote the number of representations of $\nu$ as sums of three integer squares in $K$. Then $\theta(\mathfrak{z})^{3}$ is a generating function of $r_{3,K}(\nu)$'s, namely $\theta(\mathfrak{z})^{3}=1+\sum_{\nu\succ0}r_{3,K}(\nu)\mathbf{e}(\mathrm{tr}(\nu\mathfrak{z}))$. Let $\alpha$ be a totally positive square-free integer in $K$. The Shimura lift $\mathscr{S}_{a^{\ast},\chi}(\theta(\mathfrak{z})^{3})$ is in $\mathbf{M}_{2}(\Gamma_{0}(\mathfrak{d}_{K}^{-1},2\mathfrak{d}_{K}))$, and by Table \ref{tbl:vtacst} and Table \ref{tbl:ve2ac}, it is equal to $3^{2}L_{F}(0,\widetilde{\chi}_{-4}^{F})\{3G_{2}(\mathfrak{z};\mathcal{O}_{K},\mathcal{O}_{K})+G_{2,\mathbf{1}_{\mathfrak{p}_{2}}}(\mathfrak{z};\mathcal{O}_{K},\mathcal{O}_{K})\}+c_{\alpha}\Xi(\mathfrak{z})$ in Case $(\mathrm{A})$, $3^{2}L_{F}(0,\widetilde{\chi}_{-4}^{F})\{G_{2}(\mathfrak{z};\mathcal{O}_{K},\mathcal{O}_{K})+G_{2,\mathbf{1}_{\mathfrak{p}_{2}}}(\mathfrak{z};\mathcal{O}_{K},\mathcal{O}_{K})\}+c_{\alpha}\Xi(\mathfrak{z})$ in Case $(\mathrm{B})$, $3^{2}L_{F}(0,\widetilde{\chi}_{-4}^{F})G_{2}(\mathfrak{z};\mathcal{O}_{K},\mathcal{O}_{K})+c_{\alpha}\Xi(\mathfrak{z})$ in Case $(\mathrm{C}_{1})$, and $2^{-1}3^{2}L_{F}(0,\widetilde{\chi}_{-4}^{F})$ $\{G_{2}(\mathfrak{z};\mathcal{O}_{K},\mathcal{O}_{K})-G_{2,\mathbf{1}_{\mathfrak{p}_{2}}}(\mathfrak{z};\mathcal{O}_{K},\mathcal{O}_{K})\}+c_{\alpha}\Xi(\mathfrak{z})$ in Cases $(\mathrm{C}_{2})$ and $(\mathrm{D})$ where $c_{\alpha}$ is a constant depending only on $\alpha$.
\begin{lemma} In all cases, $c_{\alpha}$ is $0$.
\end{lemma}
\begin{proof} We note that the Fourier coefficient of $\Xi(\mathfrak{z})$ for $\nu=1$ is $1$ and that for $\nu=\varepsilon_{0}$ is $-1$ by (\ref{eqn:Xi}). At first we consider Case $(\mathrm{C}_{2})$ or $(\mathrm{D})$. In this case the coefficient of $\sqrt{3}$ in $\alpha$ is odd, and hence $\alpha$ can not be represented as a sum of squares. Comparing the Fourier coefficients of the equality before the lemma, we have $r_{3,K}(\alpha)=c_{\alpha}$ since the Fourier coefficient for $\nu=1$, of $G_{2}(\mathfrak{z};\mathcal{O}_{K},\mathcal{O}_{K})-G_{2,\mathbf{1}_{\mathfrak{p}_{2}}}(\mathfrak{z};\mathcal{O}_{K},\mathcal{O}_{K})$ vanishes. Then $c_{\alpha}=0$.

Next, we consider Case $(\mathrm{A})$. Comparing the Fourier coefficients of the both sides of the equality $\mathscr{S}_{\alpha,\chi_{-4}^{K}}(\theta(\mathfrak{z})^{3})=3^{2}L_{F}(0,\widetilde{\chi}_{-4}^{F})\{3G_{2}(\mathfrak{z};\mathcal{O}_{K},\mathcal{O}_{K})+G_{2,\mathbf{1}_{\mathfrak{p}_{2}}}(\mathfrak{z};\mathcal{O}_{K},\mathcal{O}_{K})\}+c_{\alpha}\Xi(\mathfrak{z})$ for $\nu=1$ and for $\nu=\varepsilon_{0}$, we obtain $r_{3,K}(\alpha)=2^{4}3^{2}L_{F}(0,\widetilde{\chi}_{-4}^{F})+c_{\alpha}$ and $r_{3,K}(\varepsilon_{0}^{2}\alpha)=2^{4}3^{2}L_{F}(0,\widetilde{\chi}_{-4}^{F})-c_{\alpha}$. However obviously the equality $r_{3,K}(\alpha)=r_{3,K}(\varepsilon_{0}^{2}\alpha)$ holds, and hence $c_{\alpha}=0$. The similar argument shows the assertion also in the rest of cases.
\end{proof}
\begin{corollary}\label{cor:ffslt3} The Shimura lift $\mathscr{S}_{\alpha,\chi_{-4}^{K}}(\theta(\mathfrak{z})^{3})$ is equal to 
\begin{alignat*}{2}
&3^{2}L_{F}(0,\widetilde{\chi}_{-4}^{F})\{3G_{2}(\mathfrak{z};\mathcal{O}_{K},\mathcal{O}_{K})+G_{2,\mathbf{1}_{\mathfrak{p}_{2}}}(\mathfrak{z};\mathcal{O}_{K},\mathcal{O}_{K})\}&\hspace{2em}&(\mathrm{A}),\\ 
&3^{2}L_{F}(0,\widetilde{\chi}_{-4}^{F})\{G_{2}(\mathfrak{z};\mathcal{O}_{K},\mathcal{O}_{K})+G_{2,\mathbf{1}_{\mathfrak{p}_{2}}}(\mathfrak{z};\mathcal{O}_{K},\mathcal{O}_{K})\}&&(\mathrm{B}),\\
&3^{2}L_{F}(0,\widetilde{\chi}_{-4}^{F})G_{2}(\mathfrak{z};\mathcal{O}_{K},\mathcal{O}_{K})&&(\mathrm{C}_{1}),\\
 &2^{-1}3^{2}L_{F}(0,\widetilde{\chi}_{-4}^{F})\{G_{2}(\mathfrak{z};\mathcal{O}_{K},\mathcal{O}_{K})-G_{2,\mathbf{1}_{\mathfrak{p}_{2}}}(\mathfrak{z};\mathcal{O}_{K},\mathcal{O}_{K})\}&&(\mathrm{C}_{2}),(\mathrm{D}).
\end{alignat*}
\end{corollary}
Since $\mathscr{S}_{\alpha,\chi_{-4}^{K}}(\theta(\mathfrak{z})^{3})=C+\sum_{\nu\in\mathcal{O}_{K},\succ0}\sum_{\delta|\nu,\delta\in\mathcal{O}_{K}/\mathcal{O}_{K}^{\times}}(\psi_{-\alpha}\mathbf{1}_{\mathfrak{p}_{2}})(\delta)r_{3,K}(\alpha(\nu/\delta)^{2})$ $\times\mathbf{e}(\mathrm{tr}(\nu\mathfrak{z}))$ with a constant $C$, comparing the terms corresponding to $\nu$, of the both sides of equations in Corollary \ref{cor:ffslt3}, we have for $\nu\succ0$,
\begin{align}\label{eqn:cfc}
&\sum_{\delta|\nu,\delta/\mathcal{O}_{K}^{\times}}(\psi_{-\alpha}\mathbf{1}_{\mathfrak{p}_{2}})(\delta)r_{3,K}(\alpha(\nu/\delta)^{2})\\
=&2^{2}3^{2}L_{F}(0,\widetilde{\chi}_{-4}^{F})\times\begin{cases}
\{3\sigma_{1}(\nu)+\sigma_{1,\mathbf{1}_{\mathfrak{p}_{2}}}(\nu)\}&(\mathrm{A}),\\
\{\sigma_{1}(\nu)+\sigma_{1,\mathbf{1}_{\mathfrak{p}_{2}}}(\nu)\}&(\mathrm{B}),\\
\sigma_{1}(\nu)&(\mathrm{C}_{1}),\\
2^{-1}\{\sigma_{1}(\nu)-\sigma_{1,\mathbf{1}_{\mathfrak{p}_{2}}}(\nu)\}&(\mathrm{C}_{2}),(\mathrm{D}).
\end{cases}\nonumber
\end{align}By (\ref{eqn:vl}), $L_{F}(0,\widetilde{\chi}_{-4}^{F})=2^{-1}3^{-1}L_{K}(0,\psi_{-\alpha})$, and there holds $L_{K}(0,\psi_{-\alpha})=$\linebreak$\frac{2^{2}}{w_{K(\sqrt{-\alpha})}Q_{K(\sqrt{-\alpha})/K}}h_{K(\sqrt{-\alpha})}$ where $w_{K(\sqrt{-\alpha})}$, $Q_{K(\sqrt{-\alpha})/K}$ and $h_{K(\sqrt{-\alpha})}$ denote the number of units in $K(\sqrt{-\alpha})$, the Hasse unit index and the relative class number respectively. Here the relative class number is just the class number of $K(\sqrt{-\alpha})$ since $K$ is of class number $1$. If square-free $\alpha\succ0$ is not a unit, then $w_{K(\sqrt{-\alpha})}$ is $2$ and $Q_{K(\sqrt{-\alpha})/K}$ is $1$.  For $\alpha=1$, $w_{K(\sqrt{-1})}$ is $12$ and $Q_{K(\sqrt{-1})/K}$ is $2$, and for $\alpha=\varepsilon_{0}$,  $w_{K(\sqrt{-\varepsilon_{0}})}$ is $2$ and $Q_{K(\sqrt{-\varepsilon_{0}})/K}$ is $2$. Thus $L_{F}(0,\widetilde{\chi}_
{-4}^{F})$ is equal to $2^{-2}3^{-2}h_{K(\sqrt{-1})}$ for $\alpha=1$, $2^{-1}3^{-1}h_{K(\sqrt{-\varepsilon_{0}})}$ for $\alpha=\varepsilon_{0}$, and $3^{-1}h_{K(\sqrt{-\alpha})}$ for $\alpha$ non-unit.  If $\nu$ is a product of totally positive prime elements, by using the M\"obius inversion formula on $K$, it is shown that for square-free $\alpha\ne\varepsilon_{0}^{n},\succ0$, $r_{3,K}(\alpha\nu^{2})$ is equal to $2^{2}3h_{K(\sqrt{-\alpha})}$ times the following;
\begin{alignat*}{2}
&\sum_{0\prec\delta|\nu,\delta\in\mathcal{O}_{K}/\mathcal{E}_{\mathcal{O}_{K}}}(\psi_{-\alpha}\mathbf{1}_{\mathfrak{p}_{2}}\mu_{K})(\delta)\{3\sigma_{1}(\nu/\delta)+\sigma_{1,\mathbf{1}_{\mathfrak{p}_{2}}}(\nu/\delta)\}&\hspace{2em}&(\mathrm{A}),\\
&\sum_{0\prec\delta|\nu,\delta\in\mathcal{O}_{K}/\mathcal{E}_{\mathcal{O}_{K}}}(\psi_{-\alpha}\mathbf{1}_{\mathfrak{p}_{2}}\mu_{K})(\delta)\{\sigma_{1}(\nu/\delta)+\sigma_{1,\mathbf{1}_{\mathfrak{p}_{2}}}(\nu/\delta)\}&&(\mathrm{B}),\\
&\sum_{0\prec\delta|\nu,\delta\in\mathcal{O}_{K}/\mathcal{E}_{\mathcal{O}_{K}}}(\psi_{-\alpha}\mathbf{1}_{\mathfrak{p}_{2}}\mu_{K})(\delta)\sigma_{1}(\nu/\delta)&&(\mathrm{C}_{1}),\\
&2^{-1}\sum_{0\prec\delta|\nu,\delta\in\mathcal{O}_{K}/\mathcal{E}_{\mathcal{O}_{K}}}(\psi_{-\alpha}\mathbf{1}_{\mathfrak{p}_{2}}\mu_{K})(\delta)\{\sigma_{1}(\nu/\delta)-\sigma_{1,\mathbf{1}_{\mathfrak{p}_{2}}}(\nu/\delta)\}&&(\mathrm{C}_{2}),(\mathrm{D}),
\end{alignat*}
$\sigma_{1}$ and $\sigma_{1,\mathbf{1}_{\mathfrak{p}_{2}}}$ being as in (\ref{def:sigma1}). Further for $\alpha=1,\varepsilon_{0}$ and for a product $\nu$ of totally positive primes,
\begin{align*}
r_{3,K}(\nu^{2})&=h_{K(\sqrt{-1})}\sum_{0\prec\delta|\nu,\delta\in\mathcal{O}_{K}/\mathcal{E}_{\mathcal{O}_{K}}}(\psi_{-1}\mathbf{1}_{\mathfrak{p}_{2}}\mu_{K})(\delta)\{3\sigma_{1}(\nu/\delta)+\sigma_{1,\mathbf{1}_{\mathfrak{p}_{2}}}(\nu/\delta)\},\\
r_{3,K}(\varepsilon_{0}\nu^{2})&=3h_{K(\sqrt{-\varepsilon_{0}})}\sum_{0\prec\delta|\nu,\delta\in\mathcal{O}_{K}/\mathcal{E}_{\mathcal{O}_{K}}}(\psi_{-\varepsilon_{0}}\mathbf{1}_{\mathfrak{p}_{2}}\mu_{K})(\delta)\{\sigma_{1}(\nu/\delta)-\sigma_{1,\mathbf{1}_{\mathfrak{p}_{2}}}(\nu/\delta)\},
\end{align*}
where $h_{K(\sqrt{-1})}=1$ and $h_{K(\sqrt{-\varepsilon_{0}})}=2$.
\begin{theorem}\label{thm:sts} Let $K=\mathbf{Q}(\sqrt{3})$.

(i) A totally positive integer $a+b\sqrt{3}\ (a,b\in\mathbf{Z})$ is represented as a sum of three integer squares in $K$ if and only if $b$ is even.

(ii) Let $\alpha$ be a totally positive square-free integer in $K$ which is not a unit. We classify $\alpha$ as in Table \ref{tbl:qeK}. Then the class number of $K(\sqrt{-\alpha})$ is given by $2^{-4}3^{-1}r_{3,K}(\alpha)$ in Case $(\mathrm{A})$, $2^{-3}3^{-1}r_{3,K}(\alpha)$ in Case $(\mathrm{B})$, $2^{-2}3^{-1}r_{3,K}(\alpha)$ in Case $(\mathrm{C}_{1})$, and $2^{-2}3^{-2}r_{3,K}(4\alpha)$ in Cases $(\mathrm{C}_{2})$ and $(\mathrm{D})$.
\end{theorem}
\begin{proof}(i) The necessity is obvious, and only the sufficiency is to be proved. We can write as $a+b\sqrt{3}=\alpha\nu^{2}$ for $\alpha$ square-free. If $\alpha$ satisfies $(\mathrm{A}),(\mathrm{B})$ or $(\mathrm{C}_{1})$, then $r_{3,K}(\alpha)>0$ by the formulas before Theorem. Then $r_{3,K}(\alpha\nu^{2})>0$ for any $\nu\in\mathcal{O}_{K}$. Suppose that $\alpha$ satisfies $(\mathrm{C}_{2})$ or $(\mathrm{D})$. If $\nu$ is odd, then $\nu^{2}\equiv1\pmod{2}$ and $b$ is odd, which contradicts to our assumption. Hence $\nu$ must be in $\mathfrak{p}_{2}$, and $\alpha\nu^{2}$ is a product of $\pi_{2}^{2}\alpha$ and some integer square. To prove $r_{3,K}(\alpha\nu^{2})>0$, it is enough to show $r_{3,K}(\pi_{2}^{2}\alpha)>0$. 
By (\ref{eqn:cfc}) in Cases ($\mathrm{C}_{2}$) and ($\mathrm{D}$), there holds
\begin{align}\label{eqn:r3pi4}
r_{3,K}(\pi_{2}^{4}\alpha)=2^{2}3^{2}h_{K(\sqrt{-\alpha})}>0,
\end{align}
and there are $x_{1},x_{2},x_{3}\in\mathcal{O}_{K}$ satisfying $x_{1}^{2}+x_{2}^{2}+x_{3}^{2}=\pi_{2}^{4}\alpha$. At least one of $x_{i}$'s is even, say $x_{3}$. If $v_{\mathfrak{p}_{2}}(x_{3})=1$, then putting $x_{3}=\pi_{2}x_{3}'$, there holds $x_{1}^{2}+x_{2}^{2}=\pi_{2}^{2}(\pi_{2}^{2}\alpha-x_{3}'^{2})$ with $\rho_{2}(\pi_{2}^{2}\alpha-x_{3}'^{2})=1$. Then by Corollary \ref{cor:stse} (i),  both of $x_{1},x_{2}$ are even, and hence $r_{3,K}(\pi_{2}^{2}\alpha)>0$. If $x_{3}=0$ or $v_{\mathfrak{p}_{2}}(x_{3})\ge2$, then there are even $x_{1},x_{2}$ by Corollary \ref{cor:stse} (ii), and hence $r_{3,K}(\pi_{2}^{2}\alpha)>0$.

The assertion (ii) is obtained from formulas before the theorem and from (\ref{eqn:r3pi4}) since $\pi_{2}^{4}=4\varepsilon_{0}^{2}$.
\end{proof}
\begin{remark} A formula for the left hand side of (\ref{eqn:cfc}) in case $\mathrm{N}(\nu)<0$ is obtained by using the theta lifts of $\theta_{\mathfrak{d}_{K}^{-1}}(\mathfrak{z})=\sum_{\mu\in\mathfrak{d}_{K}^{-1}}\mathbf{e}(\mathrm{tr}(\mu^{2}\mathfrak{z}))$. We omit the argument because it is not necessary to show Theorem \ref{thm:sts}.
\end{remark}

\section{Hilbert modular forms on $\mathbf{Q}(\sqrt{17})$}\label{sect:hm17}
Let $K=\mathbf{Q}(\sqrt{17})$. Then $\mathfrak{d}_{K}=(\sqrt{17})$, and the class number of $K$ is $1$. We put $\omega:=(1+\sqrt{17})/2$ and $\omega':=(1-\sqrt{17})/2$. Let $\varepsilon_{0}:=3+2\omega$ be a fundamental unit, which has a negative norm. Put $\pi_{2}:=2+\omega$, $\mathfrak{p}_{2}:=(\pi_{2})$, $\pi_{2}':=2+\omega'$, $\mathfrak{p}_{2}':=(\pi_{2}')$. Then the ideal $(2)$ is decomposed as $(2)=\mathfrak{p}_{2}\mathfrak{p}_{2}'$ in $K$. Let $\chi_{-4}^{K}:=\chi_{-4}\circ\mathrm{N}\in C_{(4)}^{\ast}$, whose conductor $\mathfrak{f}_{\chi_{-4}^{K}}$ is $(4)$. Let $\rho_{2}$ be the unique nontrivial ideal class character in $C_{\mathfrak{p}_{2}^{2}}^{\ast}$, and let $\rho_{2}'$ be the unique nontrivial ideal class character in $C_{\mathfrak{p}_{2}'^{2}}^{\ast}$. Then $\mathfrak{f}_{\rho_{2}}=\mathfrak{p}_{2}^{2},\mathfrak{f}_{\rho_{2}}=\mathfrak{p}_{2}'^{2}$ and
\begin{align*}
\chi_{-4}^{K}=\rho_{2}\rho_{2}'.
\end{align*}
For $a,b\in\mathbf{Z}$, we have $\rho_{2}((a+b\omega))=\mathrm{sgn}(a+b\omega')\chi_{-4}(a),\rho_{2}'((a+b\omega))=\mathrm{sgn}(a+b\omega)\chi_{-4}(a+b)$, and $\mathbf{e}_{\rho_{2}}=(0,1),\ \mathbf{e}_{\rho_{2}'}=(1,0)$. Further $\rho_{2}(\mathfrak{p}_{2}')=-1,\rho_{2}'(\mathfrak{p}_{2})=-1$, $\tau_{K}(\chi_{-4}^{K})=-4,\tau_{K}(\rho_{2})=\tau_{K}(\rho_{2}')=2\sqrt{-1}$, $L_{K}(0,\chi_{-4}^{K})=2$.

The set $\mathcal{C}_{0}(4)$ of representatives of  cusps of $\Gamma_{0}(\mathfrak{d}_{K}^{-1},4\mathfrak{d}_{K})$ is taken as
\begin{align*}
\mathcal{C}_{0}(4)=\{\tfrac{1}{4\sqrt{17}},\tfrac{1}{2\pi_{2}'\sqrt{17}},\tfrac{1}{2\pi_{2}\sqrt{17}},\tfrac{1}{2\sqrt{17}},\tfrac{1}{\pi_{2}'^{2}\sqrt{17}},\tfrac{1}{\pi_{2}^{2}\sqrt{17}},\tfrac{1}{\pi_{2}'\sqrt{17}},\tfrac{1}{\pi_{2}\sqrt{17}},\tfrac{1}{\sqrt{17}}\},
\end{align*}
and the set $\mathcal{C}_{0}(2)$ of representatives of  cusps of $\Gamma_{0}(\mathfrak{d}_{K}^{-1},2\mathfrak{d}_{K})$ is taken as
\begin{align*}
\mathcal{C}_{0}(2)=\{\tfrac{1}{2\sqrt{17}},\tfrac{1}{\pi_{2}'\sqrt{17}},\tfrac{1}{\pi_{2}\sqrt{17}},\tfrac{1}{\sqrt{17}}\}.
\end{align*}
The square $\theta(\mathfrak{z})^{2}$ of the theta series is a Hilbert modular form for  $\Gamma_{0}(\mathfrak{d}_{K}^{-1},4\mathfrak{d}_{K})$  of weight $1$ with character $\chi_{-4}^{K}$, and it takes the values at cusps as in Table \ref{tbl:vtac17}.
\begin{table}[htb]
\caption{Values of $\theta(\mathfrak{z})^{2}$ at cusps.}
$\begin{array}{c|c@{\ }c@{\ }c@{\ }c@{\ }c@{\ }c@{\ }c@{\ }c@{\ }c}
\mathcal{C}_{0}(4)&\tfrac{1}{4\sqrt{17}}&\tfrac{1}{2\pi_{2}'\sqrt{17}}&\tfrac{1}{2\pi_{2}\sqrt{17}}&\tfrac{1}{2\sqrt{17}}&\tfrac{1}{\pi_{2}'^{2}\sqrt{17}}&\tfrac{1}{\pi_{2}^{2}\sqrt{17}}&\tfrac{1}{\pi_{2}'\sqrt{17}}&\tfrac{1}{\pi_{2}\sqrt{17}}&\tfrac{1}{\sqrt{17}}\\\hline
\theta(\mathfrak{z})^{2}&1&0&0&0&\mbox{\small$2^{-1}\sqrt{-1}$}&\mbox{\small$-2^{-1}\sqrt{-1}$}&0&0&2^{-2}
\end{array}
$\label{tbl:vtac17}
\end{table}

By (\ref{mvc1}) and (\ref{mvc2}), and by $L_{K}(0,\chi_{-4}^{K})=2$, the values at cusps, of Eisenstein series $G_{1,\chi_{-4}^{K}}(\mathfrak{z};\mathcal{O}_{K},\mathcal{O}_{K}),G_{1,\rho_{2}}^{\rho_{2}'}(\mathfrak{z};\mathcal{O}_{K},\mathcal{O}_{K})\in\mathbf{M}_{1}(\Gamma_{0}(\mathfrak{d}_{K}^{-1},4\mathfrak{d}_{K}),\chi_{-4}^{K})$ are obtained as in Table \ref{tbl:ve1ac17}.
\begin{table}[htb]
\caption{Values at cusps, of Eisenstein series of weight $1$.}
$\begin{array}{c|c@{\ }c@{\ }c@{\ }c@{\ }c@{\ }c@{\ }c@{\ }c@{\ }c}
\mathcal{C}_{0}(4)&\tfrac{1}{4\sqrt{17}}&\tfrac{1}{2\pi_{2}'\sqrt{17}}&\tfrac{1}{2\pi_{2}\sqrt{17}}&\tfrac{1}{2\sqrt{17}}&\tfrac{1}{\pi_{2}'^{2}\sqrt{17}}&\tfrac{1}{\pi_{2}^{2}\sqrt{17}}&\tfrac{1}{\pi_{2}'\sqrt{17}}&\tfrac{1}{\pi_{2}\sqrt{17}}&\tfrac{1}{\sqrt{17}}\\\hline
G_{1,\chi_{-4}^{K}}(\mathfrak{z};\mathcal{O}_{K},\mathcal{O}_{K})&2&0&0&0&0&0&0&0&2^{-1}\\
G_{1,\rho_{2}}^{\rho_{2}'}(\mathfrak{z};\mathcal{O}_{K},\mathcal{O}_{K})&0&0&0&0&\sqrt{-1}&-\sqrt{-1}&0&0&0
\end{array}
$\label{tbl:ve1ac17}
\end{table}

The product of these Eisenstein series is a cusp form of weight $2$, since it vanishes at all the cusps. Put $\Phi(\mathfrak{z}):=2^{-3}G_{1,\chi_{-4}^{K}}(\mathfrak{z};\mathcal{O}_{K},\mathcal{O}_{K})G_{1,\rho_{2}}^{\rho_{2}'}(\mathfrak{z};\mathcal{O}_{K},\mathcal{O}_{K})$, and put $\Xi(\mathfrak{z}):=3^{-1}U(2)(\Phi(\mathfrak{z}))$ which is in $\mathbf{S}_{2}(\Gamma_{0}(\mathfrak{d}_{K}^{-1},2\mathfrak{d}_{K}))$. The first several Fourier coefficients of $\Phi(\mathfrak{z}),\Xi(\mathfrak{z}),$ $U(\pi_{2})(\Phi(\mathfrak{z})),U(\pi_{2}')(\Phi(\mathfrak{z})),\Xi(2\mathfrak{z})$ are as in Table \ref{tbl:vcac17}, where  for example we read the second row as $\Phi(\mathfrak{z})=\mathbf{e}(\mathrm{tr}(\mathfrak{z}))+3\mathbf{e}(\mathrm{tr}(2\mathfrak{z}))-\mathbf{e}(\mathrm{tr}(\pi_{2}'\mathfrak{z}))-\mathbf{e}(\mathrm{tr}(\pi_{2}\mathfrak{z}))+2\mathbf{e}(\mathrm{tr}(3\mathfrak{z}))-\mathbf{e}(\mathrm{tr}((3{+}\omega')\mathfrak{z}))-\mathbf{e}(\mathrm{tr}((3{+}\omega)\mathfrak{z}))+3\mathbf{e}(\mathrm{tr}(4\mathfrak{z}))+\cdots$.
\begin{table}[htb]
\caption{Fourier coefficients of cusp forms of weight $2$.}
$\begin{array}{c||ccccccccccccc}
&1&2&\pi_{2}'&\pi_{2}&3&3+\omega'&3+\omega&4\\\hline
\Phi(\mathfrak{z})&1&3&-1&-1&2&-1&-1&3\\
\Xi(\mathfrak{z})&1&1&-1&-1&-2&-1&-1&1\\
U(\pi_{2})(\Phi(\mathfrak{z}))&-1&-3&3&1&2&1&3&-3\\
U(\pi_{2}')(\Phi(\mathfrak{z}))&-1&-3&1&3&2&3&1&-3\\
\Xi(2\mathfrak{z})&0&1&0&0&0&0&0&1\\
\end{array}
$\label{tbl:vcac17}
\end{table}

The group $\Gamma_{0}(\mathfrak{d}_{K}^{-1},4\mathfrak{d}_{K})$ acts freely on $\mathfrak{H}^{2}$, and hence the arithmetic genus of a nonsingular model of the compactified Hilbert modular surface for $\Gamma_{0}(\mathfrak{d}_{K}^{-1},4\mathfrak{d}_{K})$ is obtained from the volume of the fundamental domain and from the contributions from cusps (van der Geer \cite{Geer}). It is computed to be $6$, and hence $\dim\mathbf{S}_{2}(\Gamma_{0}(\mathfrak{d}_{K}^{-1},4\mathfrak{d}_{K}))=5$. We see that cusp forms in Table \ref{tbl:vcac17} are linearly independent only by looking the first five Fourier coefficients. This shows that $\mathbf{S}_{2}(\Gamma_{0}(\mathfrak{d}_{K}^{-1},4\mathfrak{d}_{K}))=\langle\Phi(\mathfrak{z}),$ $\Xi(\mathfrak{z}),$ $U(\pi_{2})(\Phi(\mathfrak{z})),$ $U(\pi_{2}')(\Phi(\mathfrak{z})),$ $\Xi(2\mathfrak{z})\rangle$, and that a cusp form in $\mathbf{S}_{2}(\Gamma_{0}(\mathfrak{d}_{K}^{-1},4\mathfrak{d}_{K}))$ vanishes identically if its first five Fourier coefficients are all $0$.
\begin{lemma}The following equality holds;
\begin{align}
\theta(\mathfrak{z})^{2}&=2^{-1}G_{1,\chi_{-4}^{K}}(\mathfrak{z};\mathcal{O}_{K},\mathcal{O}_{K})+2^{-1}G_{1,\rho}^{\rho'}(\mathfrak{z};\mathcal{O}_{K},\mathcal{O}_{K}).\label{eqn:sqearetheta17}
\end{align}
\end{lemma}
\begin{proof} Let $f$ be a Hilbert modular form given as the left hand side minus the right hand side of the equation (\ref{eqn:sqearetheta17}). It is a cusp form from Table \ref{tbl:vtac17} and from Table \ref{tbl:ve1ac17}. As easily checked, the first four Fourier coefficients of $f$ are all $0$, and hence so are the first five Fourier coefficients of $f^{2}\in\mathbf{S}_{2}(\Gamma_{0}(\mathfrak{d}_{K}^{-1},4\mathfrak{d}_{K}))$. Then $f^{2}=0$, and $f=0$.
\end{proof}
\begin{corollary} \label{cor:stws17} For $\nu\in\mathcal{O}_{K},\succ0$, there holds
\begin{align*}
r_{2,K}(\nu)=2\sigma_{0,\chi_{-4}^{K}}(\nu)+2\sigma_{0,\rho_{2}}^{\rho_{2}'}(\nu).
\end{align*}
\end{corollary}
In considering a prime factorization in $K$, we may assume that prime elements are totally positive since a fundamental unit has a negative norm. Let $p_{i}\in K$ be a totally positive prime element congruent to $1\pmod{4}$, $q_{j}\in K$ be a totally positive prime element congruent to $3\pmod{4}$, and let $r_{k}\in K$ be a totally positive prime element congruent to $\sqrt{17}$ or to $2+\sqrt{17}$ modulo $4$. Then a totally positive $\nu$ has a prime factorization $\nu=\varepsilon_{0}^{2l}\pi_{2}^{e}\pi_{2}'^{e'}p_{1}^{a_{1}}\cdots$ $p_{r}^{a_{r}}q_{1}^{b_{1}}\cdots q_{t}^{b_{t}}r_{1}^{c_{1}}\cdots r_{u}^{c_{u}}$. Then by Corollary \ref{cor:stws17}, it is expressed as a sum of two integer squares if and only if $e+e'+\sum_{j}b_{j}\equiv c_{1}\equiv\cdots\equiv c_{u}\pmod{2}$. The number of representations is given by $4\prod_{i}(1+a_{i})\prod_{j}(1+b_{j})$.

\section{Quadratic extensions of $\mathbf{Q}(\sqrt{17})$}
Let $\alpha\in\mathcal{O}_{K}$ be square-free and not necessarily totally positive, and let $F=K(\sqrt{\alpha})$. We denote by $\psi_{\alpha},\mathfrak{f}_{\alpha}$, the character associated the extension, the conductor of the extension. Let $\chi_{-4}^{F}:=\chi_{-4}\circ\mathrm{N}_{F}$. We classify quadratic extensions by the congruence conditions on $\alpha$ modulo some powers of $\mathfrak{p}_{2}$ and modulo some powers of $\mathfrak{p}_{2}'$ as in Table \ref{tbl:qeK17}.
\begin{table}[htb]
\caption{Quadratic extensions $F=\mathbf{Q}(\sqrt{17},\sqrt{\alpha})$ of $\mathbf{Q}(\sqrt{17})$.}
$\begin{array}{c@{\ }c@{\ }cc}
&\alpha\equiv&v_{\mathfrak{p}_{2}}(\mathfrak{f}_{\alpha})&\mathfrak{p}_{2}\mbox{ at }F\\\hline
(\mathrm{A})&1\ (\bmod\,\mathfrak{p}_{2}^{3})&0&\mbox{split}\\
(\mathrm{B})&5\ (\bmod\,\mathfrak{p}_{2}^{3})&0&\mbox{inert}\\
(\mathrm{C}_{\mathrm{A}})&7\ (\bmod\,\mathfrak{p}_{2}^{3})&2&\mbox{ramify}\\
(\mathrm{C}_{\mathrm{B}})&3\ (\bmod\,\mathfrak{p}_{2}^{3})&2&\mbox{ramify}\\
(\mathrm{D})&2\ (\bmod\,\mathfrak{p}_{2}^{2})&3&\mbox{ramify}\\\hline
\end{array}\hspace{1em}
\begin{array}{c@{\ }c@{\ }cc}
&\alpha\equiv&v_{\mathfrak{p}_{2}'}(\mathfrak{f}_{\alpha})&\mathfrak{p}_{2}'\mbox{ at }F\\\hline
(\mathrm{A}')&1\ (\bmod\,\mathfrak{p}_{2}'^{3})&0&\mbox{split}\\
(\mathrm{B}')&5\ (\bmod\,\mathfrak{p}_{2}'^{3})&0&\mbox{inert}\\
(\mathrm{C}_{\mathrm{A}'}')&7\ (\bmod\,\mathfrak{p}_{2}'^{3})&2&\mbox{ramify}\\
(\mathrm{C}_{\mathrm{B}'}')&3\ (\bmod\,\mathfrak{p}_{2}'^{3})&2&\mbox{ramify}\\
(\mathrm{D}')&2\ (\bmod\,\mathfrak{p}_{2}'^{2})&3&\mbox{ramify}\\\hline
\end{array}$\label{tbl:qeK17}
\end{table}

If $\alpha$ satisfies both $(\mathrm{A})$ and $(\mathrm{A}')$, then we say that $\alpha$ is in Case $(\mathrm{A}\mathrm{A}')$. It is similar for $(\mathrm{A}\mathrm{B}')$, $(\mathrm{A}\mathrm{C}_{\mathrm{A}'}')$ and so on.

Let $\rho_{2},\rho_{2}'$ be as in  the preceding section. Put $\rho_{2}^{F}:=\rho_{2}\circ\mathrm{N}_{F/K}, \rho_{2}'^{F}:=\rho_{2}'\circ\mathrm{N}_{F/K}$. Then since $\chi_{-4}^{K}=\rho_{2}\rho_{2}'$, we have $\chi_{-4}^{F}=\rho_{2}^{F}\rho_{2}'^{F}$ and $\widetilde{\chi}_{-4}^{F}=\widetilde{\rho}_{2}^{F}\widetilde{\rho}_{2}'^{F}$. The conductors of $\rho_{2}^{F},\rho_{2}'^{F}$ and their values of Gauss sums are given as in Table \ref{tbl:c-g17} where $\mathfrak{P}_{2},\mathfrak{P}_{2}'$ denote prime ideals in $F$ with $\mathfrak{P}_{2}^{2}=\mathfrak{p}_{2},\mathfrak{P}_{2}'^{2}=\mathfrak{p}_{2}'$ respectively.
\begin{table}[htb]
\caption{Conductors and Gauss sums of $\rho_{2}^{F},\rho_{2}'^{F}$.}
$\begin{array}{cccc}
&\mathfrak{f}_{\rho_{2}^{F}}&\widetilde{\rho}_{2}^{F}(\mathfrak{P}_{2})&\tau_{F}(\widetilde{\rho}_{2}^{F})\\\hline
(\mathrm{A})&\mathfrak{p}_{2}^{2}\mathcal{O}_{F}&\mbox{---}&-4\\
(\mathrm{B})&\mathfrak{p}_{2}^{2}\mathcal{O}_{F}&\mbox{---}&-4\\
(\mathrm{C}_{\mathrm{A}})&\mathcal{O}_{F}&-1&-1\\
(\mathrm{C}_{\mathrm{B}})&\mathcal{O}_{F}&1&-1\\
(\mathrm{D})&\mathfrak{p}_{2}\mathcal{O}_{F}&0&-2\\\hline
\end{array}
\hspace{1em}
\begin{array}{cccc}
&\mathfrak{f}_{\rho_{2}'^{F}}&\widetilde{\rho}_{2}'^{F}(\mathfrak{P}_{2}')&\tau_{F}(\widetilde{\rho}_{2}'^{F})\\\hline
(\mathrm{A'})&\mathfrak{p}_{2}'^{2}\mathcal{O}_{F}&\mbox{---}&-4\\
(\mathrm{B'})&\mathfrak{p}_{2}'^{2}\mathcal{O}_{F}&\mbox{---}&-4\\
(\mathrm{C}_{\mathrm{A}'}')&\mathcal{O}_{F}&-1&-1\\
(\mathrm{C}_{\mathrm{B}'}')&\mathcal{O}_{F}&1&-1\\
(\mathrm{D}')&\mathfrak{p}_{2}'\mathcal{O}_{F}&0&-2\\\hline
\end{array}
$\label{tbl:c-g17}
\end{table}

We have $\mathfrak{f}_{\chi_{-4}^{F}}=\mathfrak{f}_{\rho_{2}^{F}}\mathfrak{f}_{\rho_{2}'^{F}}$. The equality $\tau_{F}(\widetilde{\chi}_{-4}^{F})=\tau_{F}(\widetilde{\rho}_{2}^{F})\tau_{F}(\widetilde{\rho}_{2'}^{F})$ holds, since $\mathfrak{f}_{\rho_{2}^{F}},\mathfrak{f}_{\rho_{2}'^{F}}$ are squares in $\mathcal{O}_{F}$ and since $\rho_{2}^{F},\rho_{2}'^{F}$ are real characters. We have 
\begin{align*}
\chi_{-4}^{K}\psi_{\alpha}=\begin{cases}\psi_{-\alpha}\mathbf{1}_{2}&(\mathrm{C}_{\mathrm{A}}\mathrm{C}_{\mathrm{A}'}'),(\mathrm{C}_{\mathrm{A}}\mathrm{C}_{\mathrm{B}'}'),(\mathrm{C}_{\mathrm{B}}\mathrm{C}_{\mathrm{A}'}'),(\mathrm{C}_{\mathrm{B}}\mathrm{C}_{\mathrm{B}'}'),\\
\psi_{-\alpha}\mathbf{1}_{\mathfrak{p}_{2}}&(\mathrm{C}_{\mathrm{A}}X'),(\mathrm{C}_{\mathrm{B}}X')\mbox{\,with\,}X'\ne\mathrm{C}_{\mathrm{A}'}',\mathrm{C}_{\mathrm{B}'}',\\
\psi_{-\alpha}\mathbf{1}_{\mathfrak{p}_{2}'}&(X\mathrm{C}_{\mathrm{A}'}'),(X\mathrm{C}_{\mathrm{B}'}')\mbox{\,with\,}X\ne\mathrm{C}_{\mathrm{A}},\mathrm{C}_{\mathrm{B}},\\
\psi_{-\alpha}&(\mbox{otherwise}),
\end{cases}
\end{align*}
and 
\begin{align}
L_{F}(s,\chi_{-4}^{F})&=L_{K}(s,\chi_{-4}^{K})L_{K}(s,\psi_{-\alpha})(1-\psi_{-\alpha}(\mathfrak{p}_{2})2^{-s})(1-\psi_{-\alpha}(\mathfrak{p}_{2}')2^{-s}),\nonumber\\
L_{F}(s,\widetilde{\chi}_{-4}^{F})&=L_{K}(s,\chi_{-4}^{K})L_{K}(s,\psi_{-\alpha}).\label{eqn:vl17}
\end{align}
In particular $L_{F}(0,\chi_{-4}^{F})=0$ in Case $(\mathrm{C}_{\mathrm{A}})$ or $(\mathrm{C}_{\mathrm{A}'}')$. The standard argument shows the following lemma, and we omit the proof.
\begin{lemma}\label{lem:elem17} (i) If $\nu\equiv7\pmod{\mathfrak{p}_{2}^{3}}$ or $\nu\equiv7\pmod{\mathfrak{p}_{2}'^{3}}$, then $r_{3,K}(\nu)=0$.

(ii) The equalities $r_{3,K}(\pi_{2}^{2}\nu)=r_{3,K}(\nu)$ and $r_{3,K}(\pi_{2}'^{2}\nu)=r_{3,K}(\nu)$ hold.
\end{lemma}

In the following argument  we exclude Cases $(\mathrm{C}_{\mathrm{A}})$ and $(\mathrm{C}_{\mathrm{A}'}')$ because of Lemma \ref{lem:elem17} (i).

Let $\alpha$ be a totally positive square-free integer in $K$. By (\ref{eqn:squaretheta}), we have 
\begin{align}
\label{eqn:sltt17}&\mathscr{S}_{\alpha,\chi_{-4}^{K}}(\theta(\mathfrak{z})^{3})\\
=&2^{-1}\mathscr{S}_{\alpha,\chi_{-4}^{K}}(\theta(\mathfrak{z})G_{1,\chi_{-4}^{K}}(\mathfrak{z};\mathcal{O}_{K},\mathcal{O}_{K}))+2^{-1}\mathscr{S}_{\alpha,\chi_{-4}^{K}}(\theta(\mathfrak{z})G_{1,\rho_{2}}^{\rho_{2}'}(\mathfrak{z};\mathcal{O}_{K},\mathcal{O}_{K})), \nonumber
\end{align}
which is a Hilbert modular forms of weight $2$ for $\Gamma_{0}(\mathfrak{d}_{K}^{-1},2\mathfrak{d}_{K})$. Let $\iota:\mathfrak{H}^{2}\longrightarrow\mathfrak{H}^{4}$ be the diagonal map associated with the inclusion of $K$ into $F$. We put
\begin{align*}
&\lambda_{2,\chi_{-4}^{K}}(\mathfrak{z};(\alpha),\mathcal{O}_{K})\\
:=&\begin{cases}
U(2)(G_{1,\chi_{-4}^{F}}(\iota(\mathfrak{z}),\mathcal{O}_{F},\mathfrak{d}_{F/K}^{-1}))&(\mathrm{AA}'),(\mathrm{AB}'),(\mathrm{BA}'),(\mathrm{BB}')\\
U(\pi_{2})(G_{1,\chi_{-4}^{F}}(\iota(\mathfrak{z}),\mathcal{O}_{F},\mathfrak{d}_{F/K}^{-1}))&(\mathrm{AC}_{\mathrm{B}'}'),(\mathrm{AD}'),(\mathrm{BC}_{\mathrm{B}'}'),(\mathrm{BD}')\\
U(\pi_{2}')(G_{1,\chi_{-4}^{F}}(\iota(\mathfrak{z}),\mathcal{O}_{F},\mathfrak{d}_{F/K}^{-1}))&(\mathrm{C}_{\mathrm{B}}\mathrm{A}'),(\mathrm{C}_{\mathrm{B}}\mathrm{B}'),(\mathrm{D}\mathrm{A}'),(\mathrm{D}\mathrm{B}')\\
G_{1,\chi_{-4}^{F}}(\iota(\mathfrak{z}),\mathcal{O}_{F},\mathfrak{d}_{F/K}^{-1})&(\mbox{otherwise})
\end{cases}
\intertext{and}
&\lambda_{2,\rho_{2}}^{\rho_{2}'}(\mathfrak{z};(\alpha),\mathcal{O}_{K})\\
:=&\begin{cases}
U(2)(G_{1,\rho_{2}^{F}}^{\rho_{2}'^{F}}(\iota(\mathfrak{z}),\mathcal{O}_{F},\mathfrak{d}_{F/K}^{-1}))&(\mathrm{AA}'),(\mathrm{AB}'),(\mathrm{BA}'),(\mathrm{BB}')\\
U(\pi_{2})(G_{1,\rho_{2}^{F}}^{\rho_{2}'^{F}}(\iota(\mathfrak{z}),\mathcal{O}_{F},\mathfrak{d}_{F/K}^{-1}))&(\mathrm{AC}_{\mathrm{B}'}'),(\mathrm{AD}'),(\mathrm{BC}_{\mathrm{B}'}'),(\mathrm{BD}')\\
U(\pi_{2}')(G_{1,\rho_{2}^{F}}^{\rho_{2}'^{F}}(\iota(\mathfrak{z}),\mathcal{O}_{F},\mathfrak{d}_{F/K}^{-1}))&(\mathrm{C}_{\mathrm{B}}\mathrm{A}'),(\mathrm{C}_{\mathrm{B}}\mathrm{B}'),(\mathrm{D}\mathrm{A}'),(\mathrm{D}\mathrm{B}')\\
G_{1,\rho_{2}^{F}}^{\rho_{2}'^{F}}(\iota(\mathfrak{z}),\mathcal{O}_{F},\mathfrak{d}_{F/K}^{-1})&(\mbox{otherwise}).
\end{cases}
\end{align*}
By \cite{Tsuyumine4}, we have $\mathscr{S}_{\alpha,\chi_{-4}^{K}}(\theta(\mathfrak{z})G_{1,\chi_{-4}^{K}}(\mathfrak{z};\mathcal{O}_{K},\mathcal{O}_{K}))=2^{-2}\lambda_{2,\chi_{-4}^{K}}(\mathfrak{z};(\alpha),\mathcal{O}_{K})$, \linebreak$\mathscr{S}_{\alpha,\chi_{-4}^{K}}(\theta(\mathfrak{z})G_{1,\rho_{2}}^{\rho_{2}'}(\mathfrak{z};\mathcal{O}_{K},\mathcal{O}_{K}))=2^{-2}\lambda_{2,\rho_{2}}^{\rho_{2}'}(\mathfrak{z};(\alpha),\mathcal{O}_{K})$. We can obtain the values at cusps of Eisenstein series on $F$ by (\ref{mvc1}) and (\ref{mvc2}), which are all rational multiples of $L_{F}(0,\widetilde{\chi}_{-4}^{F})$, and hence we obtain the values at cusps of $\lambda_{2,\chi_{-4}^{K}}(\mathfrak{z};(\alpha),\mathcal{O}_{K})$, $\lambda_{2,\rho_{2}}^{\rho_{2}'}(\mathfrak{z};(\alpha),\mathcal{O}_{K})$ by using \cite{Tsuyumine4} Lemma 8.2. Here in Table  \ref{tbl:qeK17-2}, we make a rougher classification of $\alpha$ than Table \ref{tbl:qeK17}, where the cases that $\alpha\equiv7\pmod{\mathfrak{p}_{2}^{3}}$ and that $\alpha\equiv7\pmod{\mathfrak{p}_{2}'^{3}}$, are omitted. Then the equation (\ref{eqn:sltt17}) gives the values at cusps, of $\mathscr{S}_{\alpha,\chi_{-4}^{K}}(\theta(\mathfrak{z})^{3})$ as in Table \ref{tbl:vtacst17}.
\begin{table}[htb]
\caption{Classification of $\alpha$.}
$\begin{array}{cl}\hline
(\mathrm{E})&\alpha\not\equiv 3\pmod{\mathfrak{p}_{2}^{3}}\mbox{\ and\ }\alpha\not\equiv 3\pmod{\mathfrak{p}_{2}'^{3}},\\
(\mathrm{F})&\mbox{only one of two congruences\ }\alpha\equiv 3\pmod{\mathfrak{p}_{2}^{3}},\,\alpha\equiv 3\pmod{\mathfrak{p}_{2}'^{3}}\mbox{\ holds},\\
(\mathrm{G})&\alpha\equiv3\pmod{8}.\\\hline
\end{array}
$\label{tbl:qeK17-2}
\end{table}
\begin{table}[htb]
\caption{Values of $L_{F}(0,\widetilde{\chi}_{-4}^{F})^{-1}\mathscr{S}_{\alpha,\chi_{-4}^{K}}(\theta(\mathfrak{z})^{3})$ at cusps in $\mathcal{C}_{0}(2)$.}
$\begin{array}{c|ccccl}
\mathcal{C}_{0}(2)&\tfrac{1}{2\sqrt{17}}&\tfrac{1}{\pi_{2}'\sqrt{17}}&\tfrac{1}{\pi_{2}\sqrt{17}}&\tfrac{1}{\sqrt{17}}&\\\hline
\smash{\raisebox{-1.2em}{$\tfrac{\mathscr{S}_{\alpha,\chi_{-4}^{K}}(\theta(\mathfrak{z})^{3})}{L_{F}(0,\widetilde{\chi}_{-4}^{F})}$}}
&2^{-3}&-2^{-4}&-2^{-4}&2^{-5}&(\mathrm{E})\\
&2^{-2}&-2^{-3}&-2^{-3}&2^{-4}&(\mathrm{F})\\
&2^{-1}&-2^{-2}&-2^{-2}&2^{-3}&(\mathrm{G})
\end{array}
$\label{tbl:vtacst17}
\end{table}

\section{Sums of three squares in $\mathbf{Q}(\sqrt{17})$}
The group $\Gamma_{0}(\mathfrak{d}_{K}^{-1},2\mathfrak{d}_{K})$ has $4$ elliptic fixed points on $\Gamma_{0}(\mathfrak{d}_{K}^{-1},2\mathfrak{d}_{K})\backslash\mathfrak{H}^{2}$, which are all of order $2$. Since $K$ has a fundamental unit with negative norm, the arithmetic genus of a compactified Hilbert modular surface for $\Gamma_{0}(\mathfrak{d}_{K}^{-1},2\mathfrak{d}_{K})$ is obtained from the volume of fundamental domain and from contributions from elliptic singularities, and it is computed to be $2$, and hence $\dim\mathbf{S}_{2}(\Gamma_{0}(\mathfrak{d}_{K}^{-1},2\mathfrak{d}_{K}))=1$. For $\Xi(\mathfrak{z})$ in Sect.~\ref{sect:hm17}, we have $\mathbf{S}_{2}(\Gamma_{0}(\mathfrak{d}_{K}^{-1},2\mathfrak{d}_{K}))=\langle\Xi(\mathfrak{z})\rangle$.

By (\ref{mvc1}) and (\ref{mvc2}), and by using the fact $\zeta_{K}(-1)=1/3$ and $L(-1,\chi_{17})=-4$, the values at cusps, of Eisenstein series $G_{2}(\mathfrak{z};\mathcal{O}_{K},\mathcal{O}_{K}),G_{2,\mathbf{1}_{\mathfrak{p}_{2}}}(\mathfrak{z};\mathcal{O}_{K},\mathcal{O}_{K})$,\linebreak$G_{2,\mathbf{1}_{\mathfrak{p}_{2}'}}(\mathfrak{z};\mathcal{O}_{K},\mathcal{O}_{K})$, $G_{2,\mathbf{1}_{2}}(\mathfrak{z};\mathcal{O}_{K},\mathcal{O}_{K})$ are given as in Table \ref{tbl:ve2ac17}. The table shows that these Eisenstein series are linearly independent, and $\mathbf{M}(\Gamma_{0}(\mathfrak{d}_{K}^{-1},2\mathfrak{d}_{K}))=\langle G_{2}(\mathfrak{z};\mathcal{O}_{K}$, $\mathcal{O}_{K})$, $G_{2,\mathbf{1}_{\mathfrak{p}_{2}}}(\mathfrak{z};\mathcal{O}_{K},\mathcal{O}_{K})$, $G_{2,\mathbf{1}_{\mathfrak{p}_{2}'}}(\mathfrak{z};\mathcal{O}_{K},\mathcal{O}_{K})$, $G_{2,\mathbf{1}_{2}}(\mathfrak{z};\mathcal{O}_{K},\mathcal{O}_{K}),\Xi(\mathfrak{z})\rangle$.
\begin{table}[htb]
\caption{Values at cusps, of Eisenstein series of weight $2$.}
$\begin{array}{c|cccc}
\mathcal{C}_{0}(2)&\tfrac{1}{2\sqrt{17}}&\tfrac{1}{\pi_{2}'\sqrt{17}}&\tfrac{1}{\pi_{2}\sqrt{17}}&\tfrac{1}{\sqrt{17}}\\\hline
G_{2}(\mathfrak{z};\mathcal{O}_{K},\mathcal{O}_{K})&3^{-1}&3^{-1}&3^{-1}&3^{-1}\\
G_{2,\mathbf{1}_{\mathfrak{p}_{2}}}(\mathfrak{z};\mathcal{O}_{K},\mathcal{O}_{K})&-3^{-1}&2^{-1}3^{-1}&-3^{-1}&2^{-1}3^{-1}\\
G_{2,\mathbf{1}_{\mathfrak{p}_{2}'}}(\mathfrak{z};\mathcal{O}_{K},\mathcal{O}_{K})&-3^{-1}&-3^{-1}&2^{-1}3^{-1}&2^{-1}3^{-1}\\
G_{2,\mathbf{1}_{2}}(\mathfrak{z};\mathcal{O}_{K},\mathcal{O}_{K})&3^{-1}&-2^{-1}3^{-1}&-2^{-1}3^{-1}&2^{-2}3^{-1}
\end{array}
$\label{tbl:ve2ac17}
\end{table}

Let $\alpha$ be a totally positive square-free integer in $K$. The Shimura lift $\mathscr{S}_{\alpha,\chi_{-4}^{K}}(\theta(\mathfrak{z})^{3})$ is in $\mathbf{M}_{2}(\Gamma_{0}(\mathfrak{d}_{K}^{-1},2\mathfrak{d}_{K}))$, and by Table \ref{tbl:vtacst17} and Table \ref{tbl:ve2ac17}, it is equal to $2^{-3}3L_{F}(0,\widetilde{\chi}_{-4}^{F})$ $\times G_{2,\mathbf{1}_{2}}(\mathfrak{z};\mathcal{O}_{K},\mathcal{O}_{K})+c_{\alpha}\Xi(\mathfrak{z})$ in Case $(\mathrm{E})$, $2^{-2}3L_{F}(0,\widetilde{\chi}_{-4}^{F})G_{2,\mathbf{1}_{2}}(\mathfrak{z};\mathcal{O}_{K},\mathcal{O}_{K})+c_{\alpha}\Xi(\mathfrak{z})$ in Case $(\mathrm{F})$, and  $2^{-1}3L_{F}(0,\widetilde{\chi}_{-4}^{F})G_{2,\mathbf{1}_{2}}(\mathfrak{z};\mathcal{O}_{K},\mathcal{O}_{K})+c_{\alpha}\Xi(\mathfrak{z})$ in Case $(\mathrm{G})$ where $c_{\alpha}$ is a constant depending only on $\alpha$.
\begin{lemma} In all cases, $c_{\alpha}$ is $0$.
\end{lemma}
\begin{proof} We consider the case $(\mathrm{E})$. We note that the Fourier coefficients of $\Xi(\mathfrak{z})$ for $\nu=1$ and for $\nu=\pi_{2}$ are $1$ and $-1$ respectively by Table \ref{tbl:vcac17}. Comparing the Fourier coefficients of the equality $\mathscr{S}_{\alpha,\chi_{-4}^{K}}(\theta(\mathfrak{z})^{3})=2^{-3}3L_{F}(0,\widetilde{\chi}_{-4}^{F})G_{2,\mathbf{1}_{2}}(\mathfrak{z};\mathcal{O}_{K},\mathcal{O}_{K})+c_{\alpha}\Xi(\mathfrak{z})$ for $\nu=1$ and for $\nu=\pi_{2}$, we have $r_{3,K}(\alpha)=2^{-1}3L_{F}(0,\widetilde{\chi}_{-4}^{F})+c_{\alpha}$ and $r_{3,K}(\pi_{2}^{2}\alpha)=2^{-1}3L_{F}(0,\widetilde{\chi}_{-4}^{F})-c_{\alpha}$. Then Lemma \ref{lem:elem17} (ii) leads to $c_{\alpha}=0$. The similar arguments shows the assertion also in the rest of cases.
\end{proof}
\begin{corollary} The Shimura lift $\mathscr{S}_{\alpha,\chi_{-4}^{K}}(\theta(\mathfrak{z})^{3})$ is equal to
\begin{align}
\mathscr{S}_{\alpha,\chi_{-4}^{K}}(\theta(\mathfrak{z})^{3})=2^{-3}3\,c\,L_{F}(0,\widetilde{\chi}_{-4}^{F})G_{2,\mathbf{1}_{2}}(\mathfrak{z};\mathcal{O}_{K},\mathcal{O}_{K}) \label{eqn:s-e}
\end{align}
where $c=1$ in Case $(\mathrm{E})$, $c=2$ in Case $(\mathrm{F})$ and $c=2^2$ in Case $(\mathrm{G})$.
\end{corollary}

By (\ref{eqn:vl17}), $L_{F}(0,\widetilde{\chi}_{-4}^{F})=2L_{K}(0,\psi_{-\alpha})$. Let $h_{K(\sqrt{-\alpha})},w_{K(\sqrt{-\alpha})},Q_{K(\sqrt{-\alpha})/K}$ be as in Sect.~\ref{sect:sts}. Then $w_{K(\sqrt{-\alpha})}$ is $6$ if $\sqrt{-3}\in K(\sqrt{-\alpha})$, $4$ if $\sqrt{-1}\in K(\sqrt{-\alpha})$, and $2$ if otherwise. The Hasse unit index $Q_{K(\sqrt{-\alpha})/K}$ is always $1$. By $L_{K}(0,\psi_{-\alpha})=\frac{2^{2}}{w_{K(\sqrt{-\alpha})}Q_{K(\sqrt{-\alpha})/K}}h_{K(\sqrt{-\alpha})}$, $L_{F}(0,\widetilde{\chi}_{-4}^{F})$ is expressed in terms of $h_{K(\sqrt{-\alpha})}$. Applying the M\"obius inversion formula on $K$ to the equality between $\nu$-th terms of the both sides of (\ref{eqn:s-e}), we have for $\alpha$ with $K(\sqrt{-\alpha})\ne K(\sqrt{-1}),K(\sqrt{-3})$, 
\begin{align*}
r_{3,K}(\alpha\nu^{2})
=2\cdot3\,c\,h_{K(\sqrt{-\alpha})}\sum_{0\prec\delta|\nu,\delta\in\mathcal{O}_{K}/\mathcal{E}_{\mathcal{O}_{K}}}(\psi_{-\alpha}\mathbf{1}_{2}\mu_{K})(\delta)\sigma_{1,\mathbf{1}_{2}}(\nu/\delta)\hspace{1em}(\nu\in\mathcal{O}_{K},\succ0)
\end{align*}
where $c=1$ in Case $(\mathrm{E})$, $c=2$ in Case $(\mathrm{F})$ and $c=2^2$ in Case $(\mathrm{G})$ and where $\sigma_{1,\mathbf{1}_{2}}$ is as in (\ref{def:sigma1}). When $\alpha=1$ or $\alpha=3$, we have
\begin{align*}
r_{3,K}(\nu^{2})
=3h_{K(\sqrt{-1})}\sum_{\delta|\nu,\delta\succ0,\delta/\mathcal{E}_{\mathcal{O}_{K}}}(\psi_{-1}\mathbf{1}_{2}\mu_{K})(\delta)\sigma_{1,\mathbf{1}_{2}}(\nu/\delta)\hspace{1em}(\nu\in\mathcal{O}_{K},\succ0),\\
r_{3,K}(3\nu^{2})
=2^{3}h_{K(\sqrt{-3})}\sum_{\delta|\nu,\delta\succ0,\delta/\mathcal{E}_{\mathcal{O}_{K}}}(\psi_{-3}\mathbf{1}_{2}\mu_{K})(\delta)\sigma_{1,\mathbf{1}_{2}}(\nu/\delta)\hspace{1em}(\nu\in\mathcal{O}_{K},\succ0),
\end{align*}
where $h_{K(\sqrt{-1})}=2,h_{K(\sqrt{-3})}=1$. For any $\nu\in\mathcal{O}_{K},\ne0$, one of $\pm\nu,\pm\varepsilon_{0}\nu$ is totally positive and hence the formula for $r_{3,K}(\alpha\nu^{2})$ for any $\nu$ is obtained from the above formulas since $r_{3,K}(\alpha\nu^{2})=r_{3,K}(\varepsilon_{0}^{2}\alpha\nu^{2})$.

We have shown the following;
\begin{theorem}\label{thm:sts17}Let $K=\mathbf{Q}(\sqrt{17})$. Let $\pi_{2}=(5+\sqrt{17})/2,\pi_{2}'=(5-\sqrt{17})/2$, and $\mathfrak{p}_{2}=(\pi_{2}),\mathfrak{p}_{2}'=(\pi_{2}')$.

(i) A totally positive integer in $K$ is represented as a sum of three integer squares in $K$ if and only if it is not in the form $\pi_{2}^{2e}\pi_{2}'^{2e'}\mu$ with nonnegative rational integers $e,e'$ and with $\mu\equiv7\pmod{\mathfrak{p}_{2}^{3}}$ or $\mu\equiv7\pmod{\mathfrak{p}_{2}'^{3}}$.

(ii) Let $\alpha$ be a totally positive square-free integer in $K$ which is congruent to $7$ neither modulo $\mathfrak{p}_{2}^{3}$ nor modulo $\mathfrak{p}_{2}'^{3}$. Further we assume that $K(\sqrt{-\alpha})\ne K(\sqrt{-1}),K(\sqrt{-3})$. We classify $\alpha$ as in Table \ref{tbl:qeK17-2}. Then the class number of the field $K(\sqrt{-\alpha})$ is given by $2^{-1}3^{-1}r_{3,K}(\alpha)$ in Case $(\mathrm{E})$, $2^{-2}3^{-1}r_{3,K}(\alpha)$ in Case $(\mathrm{F})$, and $2^{-3}3^{-1}r_{3,K}(\alpha)$ in Case $(\mathrm{G})$.
\end{theorem}
\section{Tables of class numbers}
A tabulation for class numbers of totally imaginary quadratic extensions $F=K(\sqrt{-\alpha})$ of $K$ is made for $220$ selected values of $\alpha$ in each case of $K=\mathbf{Q}(\sqrt{3})$ and $K=\mathbf{Q}(\sqrt{17})$. 

In case $K=\mathbf{Q}(\sqrt{3})$, square-free totally positive integers $\alpha=a+b\sqrt{3}$ with $b\ge0$ are arranged in lexicographical order, where we omit $\alpha$ with $b<0$ since rings of integers in $F$ for $\alpha=a+b\sqrt{3}$ and for $\alpha=a-b\sqrt{3}$ are isomorphic to each other and they have the same class number. We omit $a+b\sqrt{3}$ from the table if there is $n\in\mathbf{Z}$ so that $a+b\sqrt{3}=\varepsilon_{0}^{2n}(a'\pm b'\sqrt{3})$ with $a'<a$. By Theorem \ref{thm:sts}, a class number $h_{F}$ is obtained from $r_{3,K}(\alpha)$ or $r_{3,K}(4\alpha)$, and $r_{3,K}(a+b\sqrt{3})$ is obtained by counting integral solutions of
\begin{align*}
&x_{1}^{2}+3y_{1}^{2}+x_{2}^{2}+3y_{2}^{2}+x_{3}^{2}+3y_{3}^{2}=a,\\
&2x_{1}y_{1}+2x_{2}y_{2}+2x_{3}y_{3}=b
\end{align*}
in terms of ordinary integral arithmetic.

In case $K=\mathbf{Q}(\sqrt{17})$, square-free totally positive integers $\alpha=a+b\omega\ (\omega=(1+\sqrt{17})/2$ with $b\ge0$ and with $\alpha\not\equiv7\pmod{\mathfrak{p}_{2}^{3}},\alpha\not\equiv7\pmod{\mathfrak{p}_{2}'^{3}}$, are arranged in lexicographical order. A tabulation is made applying  essentially the same principle as in $\mathbf{Q}(\sqrt{3})$, where we note that rings of integers in $F$ for $\alpha=a+b\omega$ and for $\alpha=a+b-b\omega$ are isomorphic to each other. By Theorem \ref{thm:sts17}, a class number $h_{F}$ is obtained from $r_{3,K}(\alpha)$, and $r_{3,K}(a+b\omega)$ is obtained by counting integral solutions of
\begin{align*}
&x_{1}^{2}+4y_{1}^{2}+x_{2}^{2}+4y_{2}^{2}+x_{3}^{2}+4y_{3}^{2}=a,\\
&2x_{1}y_{1}+y_{1}^{2}+2x_{2}y_{2}+y_{2}^{2}+2x_{3}y_{3}+y_{3}^{2}=b.
\end{align*}

\begin{table}[h]
\caption{Table of class numbers of $F=\mathbf{Q}(\sqrt{3},\sqrt{-\alpha})$}
$\alpha=a+b\sqrt{3}$, $d_{F/\mathbf{Q}(\sqrt{3})}=(m\alpha)$, $h_{F}=$ the class number of $F$.
$\begin{array}{r@{\ \,}r@{\ \,}r@{\ \,}r|r@{\ \,}r@{\ \,}r@{\ \,}r|r@{\ \,}r@{\ \,}r@{\ \,}r|r@{\ \,}r@{\ \,}r@{\ \,}r|r@{\ \,}r@{\ \,}r@{\ \,}r}
a&b&m&h_{F}&a&b&m&h_{F}&a&b&m&h_{F}&a&b&m&h_{F}&a&b&m&h_{F}\\\hline
1&0&1&1&16&7&4&6&23&4&1&4&28&13&4&6&33&1&4&28\\
2&1&4&2&17&0&1&4&23&5&4&22&29&0&1&18&33&2&2&12\\
3&1&4&2&17&1&4&16&23&6&2&10&29&1&4&30&33&4&1&14\\
4&1&4&2&17&2&2&10&23&7&4&20&29&2&2&18&33&7&4&20\\
5&0&1&2&17&3&4&22&23&8&1&10&29&3&4&28&33&8&1&4\\
5&1&4&6&17&4&1&8&23&9&4&20&29&4&1&4&33&10&2&8\\
5&2&2&2&17&5&4&14&23&10&2&10&29&5&4&36&33&11&4&16\\
6&1&4&4&17&6&2&6&23&11&4&10&29&6&2&14&33&13&4&12\\
7&0&1&2&17&7&4&12&24&1&4&16&29&7&4&22&33&14&2&8\\
7&1&4&4&17&8&1&2&24&5&4&16&29&8&1&14&33&16&1&2\\
7&2&2&2&18&1&4&16&24&7&4&16&29&9&4&28&34&1&4&24\\
7&3&4&2&18&5&4&12&24&11&4&12&29&10&2&14&34&3&4&20\\
8&1&4&10&18&7&4&8&25&1&4&12&29&11&4&28&34&5&4&16\\
8&3&4&6&19&0&1&2&25&2&2&6&29&12&1&4&34&7&4&20\\
9&1&4&4&19&1&4&10&25&3&4&12&29&13&4&28&34&9&4&16\\
9&2&2&4&19&2&2&6&25&4&1&6&29&14&2&8&34&11&4&12\\
9&4&1&2&19&3&4&8&25&5&4&12&30&1&4&16&34&13&4&12\\
10&1&4&4&19&4&1&6&25&6&2&8&30&5&4&16&34&15&4&12\\
10&3&4&4&19&5&4&8&25&7&4&12&30&7&4&20&34&17&4&8\\
10&5&4&4&19&6&2&4&25&8&1&2&30&11&4&24&35&0&1&8\\
11&0&1&2&19&7&4&10&25&9&4&12&30&13&4&16&35&1&4&44\\
11&1&4&10&19&9&4&6&25&10&2&4&31&0&1&6&35&2&2&18\\
11&2&2&6&20&1&4&14&25&11&4&6&31&1&4&16&35&3&4&40\\
11&3&4&12&20&3&4&18&25&12&1&4&31&2&2&12&35&4&1&18\\
11&4&1&4&20&5&4&20&26&1&4&36&31&3&4&14&35&5&4&40\\
11&5&4&8&20&7&4&20&26&3&4&36&31&4&1&2&35&6&2&14\\
12&1&4&12&20&9&4&18&26&5&4&28&31&5&4&18&35&7&4&24\\
12&5&4&4&21&1&4&16&26&9&4&20&31&6&2&6&35&8&1&6\\
13&0&1&4&21&2&2&8&26&11&4&12&31&7&4&20&35&9&4&22\\
13&1&4&10&21&4&1&2&26&13&4&12&31&8&1&8&35&10&2&20\\
13&2&2&2&21&5&4&16&27&2&2&8&31&9&4&16&35&11&4&44\\
13&3&4&8&21&7&4&12&27&4&1&10&31&10&2&6&35&12&1&16\\
13&5&4&4&21&8&1&6&27&5&4&20&31&11&4&12&35&13&4&28\\
13&6&2&2&21&10&2&4&27&7&4&20&31&13&4&14&35&14&2&12\\
14&1&4&16&22&1&4&12&27&8&1&2&31&14&2&6&35&15&4&28\\
14&7&4&8&22&3&4&12&27&10&2&8&31&15&4&12&35&16&1&4\\
15&1&4&8&22&5&4&12&27&11&4&12&32&1&4&26&35&17&4&18\\
15&2&2&4&22&7&4&8&27&13&4&12&32&3&4&30&36&1&4&20\\
15&4&1&2&22&9&4&8&28&1&4&12&32&5&4&48&36&5&4&32\\
15&5&4&8&22&11&4&8&28&3&4&22&32&7&4&30&36&7&4&24\\
15&7&4&8&23&0&1&12&28&5&4&14&32&9&4&32&36&11&4&16\\
16&1&4&8&23&1&4&32&28&7&4&12&32&11&4&34&36&13&4&24\\
16&3&4&6&23&2&2&12&28&9&4&14&32&13&4&16&36&17&4&16\\
16&5&4&10&23&3&4&22&28&11&4&14&32&15&4&22&37&0&1&8
\end{array}$
\end{table}

\begin{table}[h]
\caption{Table of class numbers of $F=\mathbf{Q}(\sqrt{17},\sqrt{-\alpha})$}
$\alpha=a+b\omega\ (\omega=\frac{1+\sqrt{17}}{2})$, $d_{F/\mathbf{Q}(\sqrt{17})}=(\pi_{2}^{e}\pi_{2}'^{e'}\alpha)$, $h_{F}=$ the class number of $F$.
$\begin{array}{r@{\ }r@{\ \,}r@{\ \,}r|r@{\ }r@{\ \,}r@{\ \,}r|r@{\ }r@{\ \,}r@{\ \,}r|r@{\ }r@{\ }r@{\ \,}r|r@{\ }r@{\ \,}r@{\ \,}r}
a&b&e,e'&h_{F}&a&b&e,e'&h_{F}&a&b&e,e'&h_{F}&a&b&e,e'&h_{F}&a&b&e,e'&h_{F}\\\hline
1&0&2,2&2&18&8&2,2&12&26&13&2,0&4&33&9&2,2&24&38&3&2,2&36\\
2&0&2,2&2&18&11&2,2&4&26&15&2,2&16&33&10&2,0&10&38&4&2,2&24\\
3&0&0,0&1&19&0&0,0&2&26&16&2,2&4&33&12&2,2&32&38&7&2,2&36\\
5&0&2,2&4&19&2&0,2&4&27&2&0,2&8&33&13&2,2&20&38&8&2,2&32\\
5&1&2,2&4&19&6&0,2&6&27&6&0,2&8&33&16&2,2&16&38&9&2,0&12\\
6&0&2,2&4&19&8&0,0&1&27&8&0,0&2&33&17&2,2&16&38&11&2,2&40\\
6&1&2,0&2&19&10&0,2&4&27&10&0,2&8&33&18&2,0&4&38&12&2,2&36\\
6&3&2,2&4&21&0&2,2&16&27&14&0,2&4&33&20&2,2&12&38&15&2,2&28\\
7&3&0,2&2&21&1&2,2&20&27&16&0,0&1&34&3&2,2&24&38&16&2,2&28\\
9&1&2,2&8&21&4&2,2&12&29&0&2,2&36&34&4&2,2&36&38&17&2,0&10\\
9&2&2,0&2&21&5&2,2&24&29&1&2,2&24&34&5&2,0&12&38&19&2,2&16\\
9&4&2,2&4&21&6&2,0&4&29&4&2,2&24&34&7&2,2&24&38&20&2,2&24\\
9&5&2,2&4&21&8&2,2&12&29&5&2,2&28&34&8&2,2&28&38&23&2,2&12\\
10&0&2,2&12&21&9&2,2&16&29&6&2,0&12&34&11&2,2&20&39&3&0,2&12\\
10&3&2,2&8&21&12&2,2&8&29&8&2,2&24&34&12&2,2&24&39&7&0,2&12\\
10&4&2,2&8&22&0&2,2&28&29&9&2,2&28&34&13&2,0&8&39&11&0,2&12\\
10&5&2,0&2&22&1&2,0&6&29&12&2,2&28&34&15&2,2&24&39&15&0,2&12\\
11&0&0,0&1&22&3&2,2&20&29&13&2,2&16&34&19&2,2&20&39&19&0,2&6\\
11&2&0,2&4&22&4&2,2&16&29&14&2,0&8&34&20&2,2&12&39&23&0,2&4\\
11&6&0,2&2&22&7&2,2&24&29&16&2,2&12&35&0&0,0&4&41&0&2,2&32\\
13&0&2,2&16&22&8&2,2&16&29&17&2,2&16&35&2&0,2&12&41&1&2,2&32\\
13&1&2,2&8&22&9&2,0&6&30&0&2,2&32&35&6&0,2&10&41&2&2,0&14\\
13&4&2,2&12&22&11&2,2&12&30&1&2,0&8&35&8&0,0&3&41&4&2,2&32\\
13&5&2,2&12&22&12&2,2&12&30&3&2,2&32&35&10&0,2&8&41&5&2,2&48\\
13&6&2,0&4&23&3&0,2&6&30&4&2,2&24&35&14&0,2&12&41&8&2,2&40\\
13&8&2,2&4&23&7&0,2&8&30&7&2,2&24&35&16&0,0&3&41&9&2,2&48\\
14&0&2,2&16&23&11&0,2&4&30&8&2,2&24&35&18&0,2&8&41&10&2,0&12\\
14&1&2,0&4&25&1&2,2&24&30&9&2,0&8&37&0&2,2&36&41&12&2,2&36\\
14&3&2,2&12&25&2&2,0&6&30&11&2,2&20&37&1&2,2&28&41&13&2,2&28\\
14&4&2,2&8&25&4&2,2&28&30&12&2,2&24&37&4&2,2&44&41&16&2,2&32\\
14&7&2,2&8&25&5&2,2&24&30&15&2,2&16&37&5&2,2&24&41&17&2,2&32\\
14&8&2,2&8&25&8&2,2&16&30&16&2,2&16&37&6&2,0&8&41&18&2,0&10\\
15&3&0,2&4&25&9&2,2&24&30&17&2,0&4&37&8&2,2&28&41&20&2,2&24\\
15&7&0,2&2&25&10&2,0&8&31&3&0,2&12&37&9&2,2&32&41&21&2,2&28\\
17&1&2,2&16&25&12&2,2&20&31&7&0,2&10&37&12&2,2&36&41&25&2,2&8\\
17&2&2,0&6&25&13&2,2&20&31&11&0,2&6&37&13&2,2&28&42&0&2,2&48\\
17&4&2,2&12&26&0&2,2&24&31&15&0,2&8&37&14&2,0&8&42&3&2,2&32\\
17&5&2,2&12&26&3&2,2&24&31&19&0,2&4&37&16&2,2&28&42&4&2,2&40\\
17&9&2,2&8&26&4&2,2&16&33&0&2,2&32&37&17&2,2&20&42&5&2,0&10\\
17&10&2,0&2&26&5&2,0&10&33&1&2,2&24&37&20&2,2&20&42&7&2,2&32\\
18&3&2,2&16&26&7&2,2&24&33&2&2,0&8&37&21&2,2&16&42&8&2,2&36\\
18&4&2,2&12&26&8&2,2&28&33&4&2,2&24&37&22&2,0&4&42&11&2,2&32\\
18&5&2,0&4&26&11&2,2&24&33&5&2,2&20&38&0&2,2&48&42&12&2,2&40\\
18&7&2,2&16&26&12&2,2&16&33&8&2,2&24&38&1&2,0&10&42&13&2,0&12
\end{array}$
\end{table}

\end{document}